\documentclass{article}
\usepackage{eurosym}
\usepackage{appendix}
\usepackage[T1]{fontenc}
\usepackage{epsf,epsfig,subfigure}
\usepackage{amssymb}
\usepackage{amsfonts,mathrsfs}
\usepackage{amsmath}
\usepackage[english]{babel}
\usepackage{graphicx,multirow}
\usepackage{indentfirst}
\usepackage[colorlinks=true]{hyperref}
\usepackage{bbm}
\usepackage{geometry}
\graphicspath{{./Pictures/}}
\usepackage[latin1]{inputenc}
\usepackage{listings}
\usepackage{tikz}
\usetikzlibrary{intersections}
\usetikzlibrary{matrix}

\usepackage{multirow}
\allowdisplaybreaks[4]
\usepackage[latin1]{inputenc} % accents 8 bits dans le fichier
\usepackage[T1]{fontenc}

\usepackage{paralist,graphics,epsfig,graphicx,epstopdf,mathrsfs}
\usepackage{float,color,comment,tabulary,booktabs}
\usepackage[normalem]{ulem}

\newfloat{figure}{H}{lof}
\newfloat{table}{H}{lot}
\floatname{figure}{\figurename}
\floatname{table}{\tablename}

\numberwithin{equation}{section}
\newtheorem{theorem}{Theorem}[section]

\newtheorem{assumption}[theorem]{Assumption}

\newtheorem{definition}[theorem]{Definition}

\newtheorem{lemma}[theorem]{Lemma}

\newtheorem{proposition}[theorem]{Proposition}

\numberwithin{equation}{section}
\newenvironment{proof}[1][Proof]{\noindent\textit{#1.} }{\hfill \rule{0.5em}{0.5em}}

\newcommand\dd{\mathrm{d}}

\newcommand \dis {\displaystyle}

\newcommand{\dy}{\mathrm{d} y}

\newcommand{\ds}{\mathrm{d} s}

\newcommand{\dl}{\mathrm{d} l}
\renewcommand{\d}{\mathrm{d}}
\newcommand{\ee}{\mathrm{e}}

\newcommand{\R}{\mathbb{R}}

\newcommand{\BUC}{\rm BUC}
\newcommand{\LL}{\rm L}

\begin{document}

\title{\textbf{Traveling waves with continuous profile for hyperbolic Keller-Segel equation}}
\author{\textsc{Quentin Griette$^{(a)}$\thanks{Corresponding author. e-mail: \href{mailto:quentin.griette@univ-lehavre.fr}{quentin.griette@univ-lehavre.fr}. Q.G. acknowledges support from ANR via the project Indyana under grant agreement ANR-21-CE40-0008. } , Pierre Magal$^{(b)}$\thanks{Pierre Magal was at the origin of this research and the corresponding author of this article for the original submission. Sadly, he passed away before the reviewing process was complete. }, and Min Zhao$^{(c)}$\thanks{The research of this author is supported by  Natural Science Foundation of Tianjin (No. 23JCQNJC01010) and China Scholarship Council. e-mail: \href{mailto:minzhao9216@163.com}{minzhao9216@163.com}}}\\
	{\small \textit{$^{(a)}$Université Le Havre Normandie, Normandie Univ., LMAH UR 3821, 76600 Le Havre, France.}} \\
	{\small \textit{$^{(b)}$Univ. Bordeaux, IMB, UMR 5251, F-33400 Talence, France.}} \\
	{\small \textit{CNRS, IMB, UMR 5251, F-33400 Talence, France.}}\\
	{\small \textit{$^{(c)}$Tianjin Chengjian University, School of Science, 300384 Tianjin, China}}\\
}
\maketitle
	\begin{abstract}
	This work describes a hyperbolic model for cell-cell repulsion with population dynamics. We consider the pressure produced by a population of cells to describe their motion. We assume that cells try to avoid crowded areas and prefer locally empty spaces far away from the carrying capacity. Here, our main goal is to prove the existence of traveling waves with continuous profiles. This article complements our previous results about sharp traveling waves. We conclude the paper with numerical simulations of the PDE problem, illustrating such a result. An application to wound healing also illustrates the importance of traveling waves with a continuous and discontinuous profile. 
\end{abstract}

	\section{Introduction}
	
\noindent 	\textbf{The model and its motivation: } 	In this paper, we mainly consider the following equation:
	\begin{equation}\label{1.10}
		\left\{\begin{array}{l}
			\partial_{t} u(t, x)= \underset{\text{Cell-cell repulsion}}{\underbrace{  \chi \; \partial_{x}\left(u(t, x) \partial_{x} p(t, x)\right)}}+  \underset{\text{Vital dynamic}}{\underbrace{ \lambda \; u(t, x)\left(1- \dfrac{u(t,x)}{\kappa}\right)}}, \quad t>0, x \in \mathbb{R},\vspace{0.2cm} \\
			p(t, x)-\sigma^{2} \, \partial_{x x} p(t, x)=u(t, x), \quad t>0, x \in \mathbb{R},
		\end{array}\right.
	\end{equation}
	with the initial distribution 
	\begin{equation} \label{1.2}
		u(0,x)=u_0(x) \in \LL^\infty \left(\R\right), 
	\end{equation}
	where $\lambda>0$ is the growth rate, $\kappa>0$ is the carrying capacity, $\chi>0$ is the dispersion coefficient, $\sigma>0$  is a sensing coefficient, $x \to u(t,x)$ is the density of population, and $p(t, x)$ is an external pressure. 
	
	Here the term density of population means that 
	$$
	\int_{x_1}^{x_2} u(t,x)dx
	$$
	is the number of individuals between $x_1$ and $x_2$ (when $x_1<x_2$).

	In the model the term $\chi \partial_{x}\left(u(t, x) \partial_{x} p(t, x)\right)$  describes the cell-cell repulsion,  and a logistic term $\lambda u(t, x)(1- u(t,x)/\kappa)$   corresponds  the cell division, cell mortality, and the quadratic term $u(t,x)^2/\kappa$ corresponds to growth limitations due to quorum sensing (for short slow down the process of cell division) and due to competition for resources.  
	
	Replacing $u(t,x)$ and $p(t,x)$ by $\widehat{u}(t,x)=u(t/\lambda,x)/\kappa$ and $\widehat{p}(t,x)=p(t/\lambda,x)/ \kappa$, we obtain (dropping the hat notation)
		\begin{equation}\label{1.1}
		\left\{\begin{array}{l}
			\partial_{t} u(t, x)= \chi    \; \partial_{x}\left(u(t, x) \partial_{x} p(t, x)\right)+   u(t, x)(1- u(t, x)), \quad t>0, x \in \mathbb{R},\vspace{0.2cm} \\
			p(t, x)-\sigma^{2} \, \partial_{x x} p(t, x)=u(t, x), \quad t>0, x \in \mathbb{R}.
		\end{array}\right.
	\end{equation}
Therefore, through the paper we will assume  that 
$$
\lambda=1 \text{ and }\kappa=1.
$$
Our original motivation comes from the description of cells motion in a Petri dish. In a previous paper \cite{FGM-20}, we derived a two-dimensional version of \eqref{1.1} to model the cell-cell repulsion in a Petri dish. We considered that cells grow in a circular domain (the Petri dish) and generate a repulsive gradient that pushes back neighboring cells. We built a numerical simulation framework to study the solutions of the partial differential equation and compared the results to some real experiments realized by Pasquier and collaborators \cite{Pas-Mag-Bou-Web-11}. When starting from an isolated disk-like islet, the solution of the PDE looks like an expanding disk whose radius seems to be growing at a constant speed. We can study the shape of an expanding islet by considering traveling waves for \eqref{1.1}. Previously, we studied the well-posedness of the problem \eqref{1.1} in \cite{FGM-21A} and proved the existence of an asymptotic propagation discontinuous profile --- a traveling wave --- in \cite{FGM-21B}, corresponding to an initial data that is equal to 0 outside of some bounded region. In other words, in \cite{FGM-21B}, we considered the case of an initial population of cells with compact support: no cell exists initially outside of the islet. The traveling waves constructed in \cite{FGM-21B}  are called \textit{sharp} because the transition between the occupied space (the area where $u(t, x)>0$) and the empty space (when $u(t, x)=0$) occurs at some finite position.  We also proved in \cite{FGM-21B} that sharp traveling waves are necessarily discontinuous. Our model is related to the study of Ducrot et al. \cite{ducrot2011vitro} who introduced a complete model of in-vitro cell dynamics with many different behaviors at the cellular level. Other features of closely related models have been investigated in \cite{Duc-Fu-Mag-18, Duc-Mag-14, ducrot2020one, Hamel2020,  Henderson}.

	In the previous paper, we proved the existence of sharp traveling waves for  \eqref{1.1}. Our goal here is to complete the description of existing traveling waves that are not sharp. Formally, our work relates to the result of de Pablo and Vazquez \cite{Pablo-Vazquez}, who studied the existence of sharp and not sharp traveling waves for a porous medium equation. The porous medium equation corresponds (formally) to the case $\sigma \to 0$. 
	The convergence of the traveling waves when $\sigma \to 0$ has been observed only numerically in \cite{FGM-21B} and proved in \cite{GrietteHendersonTuranova}, where the authors also propose an alternate method for the construction of discontinuous fronts by vanishing viscosity among other results.

	\medskip
	\noindent \textbf{The notion of solution:}
	{Throughout this paper we impose that $p\in\LL^\infty(\mathbb{R})$ so the second line of \eqref{1.1} has a unique solution for a given $u(t, \cdot)\in\LL^\infty(\mathbb{R})$.}
	In order to give a sense of the solution \eqref{1.1}, we first assume that $x \to p(t,x)$ is regular enough. Then the nonlinear diffusion can be understood as
	$$
	\chi \partial_{x}\left(u(t, x) \partial_{x} p(t, x)\right)=\chi \partial_{x}u(t, x) \partial_{x} p(t, x)+\chi u(t, x) \partial_{xx} p(t, x)
	$$
	and by using the second equation of \eqref{1.1} we obtain
	$$
	\chi \partial_{x}\left(u(t, x) \partial_{x} p(t, x)\right)=\chi \partial_{x}u(t, x) \partial_{x} p(t, x)+ \frac{\chi}{\sigma^{2}} u(t, x) \left[ p(t,x)-u(t,x) \right].
	$$
	Therefore, the system \eqref{1.1} is understood for $t \geq 0$ and $x \in \R$ as
	\begin{equation}\label{1.3}
		\left\{\begin{array}{l}
			\partial_{t} u(t, x)=\chi \partial_{x}u(t, x) \partial_{x} p(t, x)+u(t, x) \left( \left( 1+\frac{\chi}{\sigma^{2}} p(t,x) \right)-\left(1+\frac{\chi}{\sigma^{2}} \right)u(t, x) \right), \vspace{0.2cm}  \\
			p(t, x)-\sigma^{2} \partial_{x x} p(t, x)=u(t, x),
		\end{array}\right.
	\end{equation}
	with the initial distribution 
	\begin{equation} \label{1.4}
		u(0,x)=u_0(x) \in \LL^\infty \left(\R\right).
	\end{equation}
	The existence and uniqueness of solutions of \eqref{1.3} in $ \LL^\infty \left(\R\right)$ have been considered as a subset of the weighted space $L^1_\eta  \left(\R\right)$ (with $\eta>0$) with the norm
	$$
	\Vert u\Vert_{L^1_\eta }=\int_{\R} e^{-\eta \vert x \vert}\vert u(x) \vert dx. 
	$$
	The existence and uniqueness of solutions for \eqref{1.3} has been studied by Fu, Griette, and Magal \cite[Theorem 2.2]{FGM-21A}.

\medskip 		
\noindent 	\textbf{Notion of traveling wave: } 
\begin{definition}\label{DEF2.1}
    A \textbf{traveling wave} is a special solution  of \eqref{1.1} such that $u(t, x)$ has the specific form 
\[ u(t, x)=U(x-ct), \text{ for a.e. } (t, x)\in\mathbb R^2, \]
where the \textbf{profile} $U$ has the following behavior at $\pm\infty$:
\[ 
\lim_{z\to - \infty}U(z) =1,\quad \lim_{z\to \infty} U(z)=0.  
\]
A traveling wave is \textbf{sharp} if  there exists $x_0 \in \R,$  such that 
$$ 
U(x)=0, \text{ for all } x >x_0.  
$$
A  traveling wave is \textbf{not sharp}  if 
$$ 
U(x)>0, \text{ for all } x \in \mathbb R. 
$$
    {We will say that system \eqref{1.1} has a \textbf{traveling wave with continuous profile} if we can find a bounded, continuous, and decreasing continuous  function $U: \R \to \R$ that is the profile of a traveling wave.}
\end{definition}
In Fu, Griette, and Magal \cite[Proposition 2.4]{FGM-21B}, we proved that the sharp traveling waves must be discontinuous.  That is to say that $x \to U(x)$   the traveling wave profile of \eqref{1.1} can be either continuous or discontinuous. We illustrated both situations in Figure \ref{Fig1}. 
	\begin{figure}[H]
		\centering
		
		\begin{tikzpicture}[line width=1pt, x=0.8cm, y=0.8cm]
			
			\draw (-7, 4.5) ..controls +(3, 0) and +(-1, 1) .. (-2, 3) -- (-2, 0);
			\path (-6.5, 2.5) node (line1) [anchor=north west] {discontinuous};
			\path (line1.south west) node [anchor=north west]{traveling wave};
			\draw[->] (-7, 0) -- (-1, 0) node [pos=0.5, below=5pt] {(a)};
			\draw[->] (-7, 0) -- (-7, 5);
			
			\draw (1, 4.5) ..controls +(3, 0) and +(-0.5, 2) .. (5, 2) ..controls +(0.5, -2) and +(-1, 0).. (6.5, 0);
			\path (1.5, 2.5) node (line1) [anchor=north west] {smooth};
			\path (line1.south west) node [anchor=north west] {traveling wave};
			\draw[->] (1, 0) -- (7, 0) node[pos=0.5, below=5pt] {(b)};
			\draw[->] (1, 0) -- (1., 5);
		\end{tikzpicture}
		\caption{\textit{An illustration of two types of traveling wave solutions.}}
		\label{Fig1}
	\end{figure}
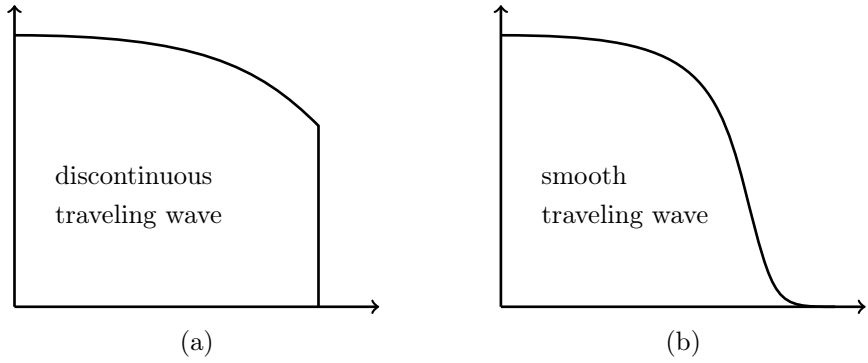

\medskip 	
\noindent 	\textbf{Estimations on the traveling speed for the discontinuous profile: } 	
Under a technical assumption on $\hat\chi=\frac{\chi}{\sigma^2}$, we can prove the existence of sharp traveling waves which present a jump at the vanishing point. 
\begin{assumption}[Bounds on $\hat\chi$]\label{ASS1.2}
	Let $\chi>0$ and $\sigma>0$ be given and define  $\hat\chi :=\frac{\chi}{\sigma^2}$. We assume that $0<\hat\chi<\bar\chi$, where $\bar\chi$ is the positive unique root of the function 
	\begin{equation*}
		\hat\chi\mapsto \ln\left(\frac{2-\hat\chi}{\hat\chi}\right)+\frac{2}{2+\hat\chi}\left(\frac{\hat\chi}{2}\ln\left(\frac{\hat\chi}{2}\right)+1-\frac{\hat\chi}{2}\right).
	\end{equation*}

\end{assumption}
The existence of traveling waves with discontinuous profile has been studied in  Fu, Griette, and Magal \cite[Theorem 2.4]{FGM-21B}. 
\begin{theorem}[Existence of a sharp discontinuous traveling wave]\label{TH1.2}
	Let Assumption \ref{ASS1.2} be satisfied. There exists a traveling wave $u(t,x)=U(x-ct)$ traveling at speed 
	\begin{equation*}
		c\in\left(\frac{\sigma\hat\chi}{2+\hat\chi},\frac{\sigma\hat\chi}{2}\right),  
	\end{equation*}
	where 
	$$
	\hat\chi=\frac{\chi}{\sigma^2}.
	$$
	
	Moreover, the profile $U$ satisfies the following properties (up to a shift in space): 
	\begin{itemize} 
		\item[{\rm (i)}] $U$ is {\em sharp} in the sense that $U(x)=0$ for all $x\geq 0$; moreover, $U$ has a discontinuity at $x=0$ with $ U(0^-)\geq \frac{2}{2+\hat\chi} $.
		\item[{\rm (ii)}]  $U$ is continuously differentiable  and strictly decreasing on $(-\infty, 0]$, and satisfies
		$$
		 -c\,U'-\chi(UP')' =U(1-U)  \text{ on } (-\infty, 0),
		 $$
		 and
		 $$
		 U=0  \text{ on } (0, \infty),
		 $$
		and
		$$
		P-\sigma^2 P''=U \text{ on }\R.
		$$
	\end{itemize}
\end{theorem}

\medskip 	
In this article, we focus on the existence of traveling waves with continuous profiles. The main result of this paper is the following theorem. 
\begin{theorem}[Existence of a continuous traveling wave]\label{TH1.3}
		We assume that
	$$
	 c\geq \sqrt{\chi\left( 1+\frac{\chi}{\sigma^2} \right)}.
	$$
	There exists a traveling wave $u(t, x)=U(x-ct)$ with a continuous profile $x \to U(x)$  is continuously differentiable  and strictly decreasing,  and 
		\begin{equation}
		\lim_{x\to -\infty}U(x)=1, \text{ and }\lim_{x\to +\infty}U(x)=0,
	\end{equation}
	and satisfies traveling wave problem 
	\begin{equation}  \label{1.6}
	-c\,U'-\chi(UP')' =U(1-U),  \text{ on }\R, 
	\end{equation}
where 
\begin{equation} \label{1.7}
	P-\sigma^2 P''=U ,  \text{ on }\R .
\end{equation}

\end{theorem}

\noindent 	\textbf{Estimations on the traveling speed: }  	We obtain the following condition for the existence of a traveling wave with a continuous profile for all-speed 
$$
 c\geq c_{\text{cont}}^*:=\sqrt{\chi\left( 1+\frac{\chi}{\sigma^2} \right)}.
$$
For the Fisher-KPP equation \cite{Aronson-Weinberger, KPP, Weinberger1982}, traveling waves only exist for half-line of positive traveling speeds. Moreover, there is a minimum speed $c_*>0$ below which no traveling wave exists. Moreover, we can construct traveling waves for any values $c$ above $c_*$. The existence of minimum speed is also true for porous medium equations with logistic dynamics \cite{Pablo-Vazquez}. By analogy with the porous medium equations, we expect that the minimal speed of the traveling waves corresponds to the sharp traveling wave constructed in \cite{FGM-21B}. In contrast, the continuous traveling waves constructed in the present paper correspond to higher velocities. Recall from \cite[Theorem 2.4]{FGM-21B} that the sharp traveling wave is expected to travel at a speed $c_{\text{sharp}}\in \left(\frac{\chi/\sigma}{2+\chi/(\sigma^2)}, \frac{\chi}{2\sigma}\right) $ and indeed we have that 
\begin{equation*}
	c_{\text{cont}}^*=\sqrt{\chi\left( 1+\frac{\chi}{\sigma^2} \right)}\geq  \dfrac{\chi}{\sigma} >  \frac{1}{2} \frac{\chi}{\sigma}\geq c_{\text{sharp}}. 
\end{equation*}
Further analysis will be necessary to connect the gap between $c_{\text{cont}}^*$ and $c_{\text{sharp}}$ and to possibly prove the non-existence of traveling waves slower than sharp waves. Understanding the relationships between the profiles and the traveling speeds is still an open problem. 
	
\medskip 	
The paper is organized as follows. Section \ref{Section2} is devoted to preliminary results. Section \ref{Section3} presents the fixed point problem and its properties. Section \ref{Section4} is devoted to the proof of Theorem \ref{TH1.3}. In Section \ref{Section5}, we present some numerical simulations. In Section \ref{Section6},  we present an application to wound healing.

\section{Preliminary}
\label{Section2}
We are interested in traveling waves of the system \eqref{1.1}. In Figure \ref{Fig1-2}, we illustrate the continuous traveling wave profiles.
\vspace{0.5cm}
\begin{figure}
	\centering
	\begin{tikzpicture}
		\coordinate [label=left:$1$] (1) at (-5,4.1);
		\coordinate [label=left:$0$](0) at (-5,-0.1);
		\draw[color=black, line width=1.5pt] (-5,4.1)--(7,4.1) ;
		\draw [name path=curve 1][color=blue, line width=1.5pt]plot[smooth, domain =-4.8:6.7](\x,{1*(1+tanh(-0.6*(\x-2)))^2})node[above]{$U(x)$};
		\draw [color=black, line width=1.5pt][name path=A](-5,-0.1)--(7,-0.1) ;
		\draw[name path=B,  line width=1.5pt] (1,-0.1) -- (1,4.1);
			\draw [name intersections={of=A and B, by=0}];
			\node [label=-90:$0$] at (0) {};
	\end{tikzpicture}
	\caption{\textit{In this figure, we plot the traveling wave profile $x \to U(x)$.}}  \label{Fig1-2}
\end{figure}
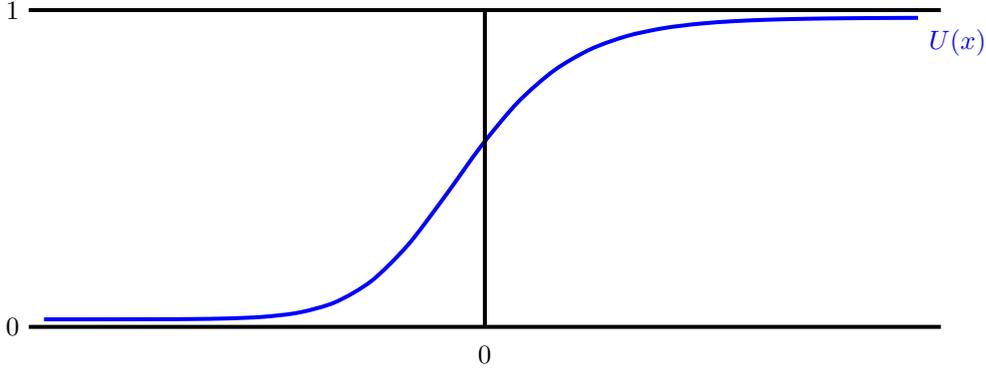

\begin{lemma}\label{LEM2.2}
	Assume that system \eqref{1.3} has a traveling wave $u(t,x)=U(x-ct)$. Then we must have
	$$
	p(t,x)=P(x-ct),
	$$
	where $P: \R \to \R$ is the unique bounded continuous function satisfying the elliptic equation
	$$
P( x)	-\sigma^{2} P''( x)=U(x), \forall x \in \mathbb{R}.
	$$
\end{lemma}
\begin{proof}
	\begin{equation}\label{2.1}
		p(t,x)=\frac{1}{2 \sigma} \int_{\mathbb{R}} e^{-\frac{|y|}{\sigma}} u(t,x-y) \d y= \frac{1}{2 \sigma} \int_{\mathbb{R}} e^{-\frac{|y|}{\sigma}} U(x-y-ct) \d y,
	\end{equation}
	therefore $p(t,x)=P(x-ct)$ where
	\begin{equation*}
		P(x)= \frac{1}{2 \sigma} \int_{\mathbb{R}} e^{-\frac{|y|}{\sigma}} U(x-y) \dy \Leftrightarrow 	P(x)-\sigma^{2} P''( x)=U(x),x \in \mathbb{R}.
	\end{equation*}
\end{proof}

The following proposition was proved in Fu, Griette and Magal \cite[Proposition 2.4]{FGM-21B}.
\begin{proposition}\label{PRO2.3}
Assume that $U$ is a continuous profile of traveling wave. Then $U:\R \to \R$ is continuously differentiable,  and
	\begin{equation} \label{2.2}
	c+\chi P'(x) >0, \forall x \in \R  \left( \Leftrightarrow 0 \leq -P'(x) < \dfrac{c}{\chi}, \forall x \in \R\right). 		
	\end{equation}
\end{proposition}

\medskip
\noindent \textbf{Transforming $U(x)$ into $\widehat{U}(x)=U(-x)$: } {In order to work with increasing functions rather than decreasing functions, we reverse the space variable.} By Definition \ref{DEF2.1} and Lemma \ref{LEM2.2}, we get the following traveling wave problem
\begin{equation}\label{2.3}
-\left(c+\chi P'(x)  \right)  U'(x)=U(x) \left( \left( 1+\frac{\chi}{\sigma^{2}} P(x) \right)-\left(1+ \frac{\chi}{\sigma^{2}}  \right)U(x) \right), \forall x \in \mathbb{R},
\end{equation}
where
\begin{equation}\label{2.4}
P( x)-\sigma^{2} P''( x)=U(x), \forall  x \in \mathbb{R}.
\end{equation}
Equation \eqref{2.3} has the following behavior at $\pm\infty$:
	$$
	\lim_{x\rightarrow-\infty}U(x) = 1,\quad \lim_{x\rightarrow +\infty}U(x) = 0.
	$$
Now, let us perform the change of variables to reverse the space direction. Setting $\widehat{U}(x) = U(-x)$, and $\widehat{P}(x) = P(-x)$, then equations \eqref{2.3} and \eqref{2.4} become
\begin{align}\label{2.5}
\left(c-\chi \widehat{P}'(x)  \right)\widehat{U}'(x)&=\widehat{U}(x)\left( \left( 1+\frac{\chi}{\sigma^{2}} \widehat{P}(x) \right)-\left(1+ \frac{\chi}{\sigma^{2}} \right)\widehat{U}(x) \right), \forall  x \in \mathbb{R},
\end{align}
where
\begin{equation*}
\widehat{P}( x)-\sigma^{2} \widehat{P}''( x)=\widehat{U}(x), \forall  x \in \mathbb{R}.
\end{equation*}
Assume that $c+\chi P'(x) >0, \forall x \in \R,$ then we have $c-\chi \widehat{P}'(x) >0, \forall x \in \R$ by using $\widehat{P}(x) = P(-x)$.

\medskip
For convenience, we drop the hat notation, and system \eqref{2.5} becomes a logistic equation
\begin{align}\label{2.6}
U'(x)&=\lambda(x) \, U(x) -\kappa(x) \, U^2(x), \forall  x \in \mathbb{R},
\end{align}
where
\begin{equation}\label{2.7}
\lambda(x):=\frac{1+\frac{\chi}{\sigma^{2}} P(x)}{c-\chi P'(x)}, \forall  x \in \mathbb{R},
\end{equation}
and
\begin{equation}\label{2.8}
\kappa(x):=\frac{1+ \frac{\chi}{\sigma^{2}}}{c-\chi P'(x)}, \forall  x \in \mathbb{R},
\end{equation}
with $P(x)$ is the unique solution of the elliptic equation
\begin{equation}\label{2.9}
		P(x)-\sigma^{2} P''(x)=U(x), \forall x \in \mathbb{R}.
	\end{equation}
System \eqref{2.6} has the following behavior at $\pm\infty$
\begin{equation}\label{2.10}
\lim_{x\rightarrow-\infty}U(x) = 0,\quad \lim_{x\rightarrow\infty}U(x) = 1.
\end{equation}
\begin{lemma}\label{LEM2.4}
	Assume that  $U:\R \to \R$ is an increasing $C^1$ function.  Then the map $x \to P(x)$ solving the elliptic equation
	\begin{equation*}
		P( x)-\sigma^{2} P''( x)=U(x), \forall x \in \mathbb{R},
	\end{equation*}
	 is an increasing $C^3$ function, and we have the following estimation of the first derivative of $P(x)$
	\begin{equation}\label{2.11}
 \sup_{x \in \R} P'(x)  \leq  \sup_{x \in \R} U'(x).
\end{equation}
\end{lemma}
\begin{proof}  The result follows the following inequality
	\begin{equation}\label{2.12}
		0 \leq 	P'(x)=\frac{1}{2 \sigma} \int_{\mathbb{R}} \ee^{-\frac{|y|}{\sigma}}U'(x-y) \dy \leq  \sup_{x \in \R} U'(x).
	\end{equation}
\end{proof}
\begin{lemma}\label{LEM2.5}
    Assume that $c\geq\sqrt{\chi\left(1+\frac{\chi}{\sigma^2}\right)}$, 
    \begin{equation}\label{eq:CU}
	0\leq U'(x)\leq {C_U:=\dfrac{c+\sqrt{c^2-\chi\left(1+\frac{\chi}{\sigma^2}\right)}}{2\chi}}, \forall x \in \mathbb{R},
    \end{equation}
and
$$
    0<u_0<\frac{\sigma^2}{2 \left( \sigma^2+\chi\right)}{\left(1-\sqrt{1-\frac{\chi}{c^2}\left(1+\frac{\chi}{\sigma^2}\right)}\right)}.
$$
Then  $U:\R \to \R$ is an increasing $C^1$ function  satisfying  \eqref{2.6}, \eqref{2.10} and
\begin{equation}\label{2.13}
U(0)=u_0,
\end{equation}
which is given by the following formula
\begin{equation*}
U(x)=\frac{u_0\ee^{\int_{0}^{x}\lambda(s)\ds}}{1+u_0\int_{0}^{x}\kappa(s)\ee^{\int_{0}^{s}\lambda(\tau)\dl}\ds}, \forall x \in \R,
\end{equation*}
where $\lambda(x)$, $\kappa(x)$, and $P(x)$ are given by equations \eqref{2.7}, \eqref{2.8}, and \eqref{2.9} above.
\end{lemma}
\begin{proof}
Let us prove that the formula
\begin{align}\label{2.14}
\frac{u_0\ee^{\int_{0}^{x}\lambda(s)\ds}}{1+u_0\int_{0}^{x}\kappa(s)\ee^{\int_{0}^{s}\lambda(l)\dl}\ds},
\end{align}
is well defined for all $x \in \R$. So let us prove that if $0<u_0<\frac{\sigma^2}{2 \left( \sigma^2+\chi\right)}{\left(1-\sqrt{1-\frac{\chi}{c^2}\left(1+\frac{\chi}{\sigma^2}\right)}\right)}$, then we have
$$
1-u_0\int_{-\infty}^{0}\kappa(s)\ee^{-\int_{s}^{0}\lambda(l)\dl}\ds>0.
$$
Indeed, since by assumption
\begin{equation*}
0\leq U'(x)\leq C_U, \; \forall x\in\R,
\end{equation*}
and Lemma \ref{LEM2.4}, we deduce that
\begin{equation}\label{2.15}
		0 \leq P'(x)\leq  \sup_{x \in \R}U'(x)\leq C_U,\;  \forall x\in\R,
	\end{equation}
hence
$$
    c-\chi P'(x)\geq c-\chi C_U=\dfrac{c-\sqrt{c^2-\chi\left(1+\frac{\chi}{\sigma^2}\right)}}{2}>0,\; \forall x\in\R,
$$
and since $P'(x)\geq 0$, we deduce that 
\begin{equation}\label{2.16}
\frac{1}{c} \leq \frac{1}{c-\chi P'(x)}\leq\dfrac{2}{c-\sqrt{c^2-\chi\left(1+\frac{\chi}{\sigma^2}\right)}}, \; \forall x\in\R.
\end{equation}
Now by combining \eqref{2.16}, and $0\leq P(x)\leq 1,$ for any $x\in\R$, we deduce that
\begin{equation}\label{2.17}
    \frac{1}{c}\leq \lambda(x)\leq\dfrac{2\left(1+\frac{\chi}{\sigma^2}\right)}{c-\sqrt{c^2-\chi\left(1+\frac{\chi}{\sigma^2}\right)}}\text{ and }0<\frac{1}{c}\left(1+\frac{\chi}{\sigma^2}\right) \leq \kappa(x)\leq\dfrac{2\left(1+\frac{\chi}{\sigma^2}\right)}{c-\sqrt{c^2-\chi\left(1+\frac{\chi}{\sigma^2}\right)}},  \forall x\in\R.
\end{equation}
Define
$$
G:=1-u_0\int_{-\infty}^{0}\kappa(s)\ee^{-\int_{s}^{0}\lambda(l)\dl}\ds.
$$
Using \eqref{2.17}, we have that
\begin{equation}\label{2.18}
\begin{aligned}
G&\geq1-u_0 \displaystyle \int_{-\infty}^{0}\dfrac{2\left(1+\frac{\chi}{\sigma^2}\right)}{c-\sqrt{c^2-\chi\left(1+\frac{\chi}{\sigma^2}\right)}}\ee^{- \textstyle\int_{s}^{0}\frac{1}{c}\dl}\ds\\
&=1-u_0\dfrac{2\left(1+\frac{\chi}{\sigma^2}\right)}{c-\sqrt{c^2-\chi\left(1+\frac{\chi}{\sigma^2}\right)}} \displaystyle \int_{-\infty}^{0}\ee^{\frac{s}{c}}\ds\\
&=1-u_0\dfrac{2c\left(1+\frac{\chi}{\sigma^2}\right)}{c-\sqrt{c^2-\chi\left(1+\frac{\chi}{\sigma^2}\right)}}>0,
\end{aligned}
\end{equation}
by using the assumption
$$
    0<u_0<\frac{\sigma^2}{2(\sigma^2+\chi)}\left(1-\sqrt{1-\frac{\chi}{c^2}\left(1+\frac{\chi}{\sigma^2}\right)}\right)=\dfrac{c-\sqrt{c^2-\chi\left(1+\frac{\chi}{\sigma^2}\right)}}{2c\left(1+\frac{\chi}{\sigma^2}\right)}.
$$
\end{proof}

\section{The relationship between the fixed point and traveling waves}
\label{Section3}
\noindent

\begin{definition}\label{DEF3.1}
	Let $\mathcal{A}$ be the set of all admissible function $U:\R\rightarrow[0, 1]$ satisfying
\begin{itemize}
  \item [{\rm (i) }] $U \in C^1(\R)$;
  \item [{\rm (ii) }] $0 \leq U(x) \leq 1, \forall x \in \mathbb{R}$;
  \item [{\rm (iii) }] $0\leq U'(x)\leq{C_U=\dfrac{c+\sqrt{c^2-\chi\left(1+\frac{\chi}{\sigma^2}\right)}}{2\chi}}, \forall x \in \mathbb{R}$.
\end{itemize}
\end{definition}
{Note that the upper bound in (iii) is the same as in the statement of Lemma \ref{LEM2.5} \eqref{eq:CU}}.

For each $U\in \mathcal{A}$, we define
\begin{equation}\label{3.1}
\mathcal{T}(U)(x):=V(x), \forall  x \in \mathbb{R},
\end{equation}
where
\begin{equation}\label{3.2}
V(x)=\frac{u_0\ee^{\int_{0}^{x}\lambda(s)\ds}}{1+u_0\int_{0}^{x}\kappa(s)\ee^{\int_{0}^{s}\lambda(l)\dl}\ds}, \forall  x \in \mathbb{R},
\end{equation}
with
\begin{equation}\label{3.3}
0<u_0<\frac{\sigma^2}{2 \left( \sigma^2+\chi\right)}\left(1-\sqrt{1-\frac{\chi}{c^2}\left(1+\frac{\chi}{\sigma^2}\right)}\right),
\end{equation}
and
\begin{equation}\label{3.4}
\lambda(x)=\frac{1+\frac{\chi}{\sigma^{2}}P(x) }{c-\chi P'(x)}, \forall  x \in \mathbb{R},
\end{equation}
and
\begin{equation}\label{3.5}
\quad\kappa(x)=\frac{1+ \frac{\chi}{\sigma^{2}}}{c-\chi P'(x)}, \forall  x \in \mathbb{R},
\end{equation}
and $P(x)$ is the unique solution of the elliptic equation
\begin{equation}\label{3.6}
		P(x)-\sigma^{2} P''(x)=U(x), \forall x \in \mathbb{R}.
	\end{equation}
\begin{assumption} \label{ASS3.2}
	We assume that
	$$
	c\geq {\sqrt{\chi\left( 1+\frac{\chi}{\sigma^2} \right)}}\text{ and }
	0<u_0<{\frac{\sigma^2}{2 \left( \sigma^2+\chi\right)}\left(1-\sqrt{1-\frac{\chi}{c^2}\left(1+\frac{\chi}{\sigma^2}\right)}\right)}.
	$$
\end{assumption}

\begin{lemma}[Invariance of $\mathcal{A}$ by $\mathcal{T}$]\label{LEM3.3}
Let Assumption \ref{ASS3.2} be satisfied. Let $\mathcal{T}$ be the map defined by \eqref{3.1}. Then
$$\mathcal{T}(\mathcal{A})\subset \mathcal{A}.$$
\end{lemma}
\begin{proof}
    We divide the proof in {three steps}. \\
	\textbf{Step 1.} We prove that $V=\mathcal{T}(U)\in C^1(\R)$. Indeed, $V$ is continuously differentiable 
	and
\begin{equation}\label{3.7}
V'(x)=\frac{\lambda(x) u_0\ee^{\int_{0}^{x}\lambda(s)\ds} (1+u_0\int_{0}^{x}\kappa(s)\ee^{\int_{0}^{s}\lambda(l)\dl}\ds) - \kappa(x) \, (u_{0}\ee^{\int_{0}^{x}\lambda(s)\ds})^2}{(1+u_0\int_{0}^{x}\kappa(s)\ee^{\int_{0}^{s}\lambda(l)\dl}\ds)^2},
\end{equation}
hence
\begin{equation}\label{3.8}
	V'(x)= \lambda(x)V(x)- \kappa(x)V^2(x), \forall  x \in \mathbb{R}.
\end{equation}
It follows from the definitions of $\lambda(x)$ and $\kappa(x)$ (see \eqref{3.4} and \eqref{3.5}), that $\lambda(x)$ and $\kappa(x)$ are continuously differentiable. Therefore, we have
$$
V \in C^1(\R).
$$
\textbf{Step 2.} We prove that $0< V(x)\leq1$, $\forall x\in\R$. By \eqref{2.11} and $U\in\mathcal{A}$, we have that
\begin{equation*}
0\leq \sup_{x \in \R} P'(x)  \leq  \sup_{x \in \R} U'(x)\leq C_U.
\end{equation*}
Therefore, we have $c-\chi P'(x)>0$. Recall that
\begin{equation*}
V(x)=\frac{u_0\ee^{\int_{0}^{x}\lambda(s)\ds}}{1+u_0\int_{0}^{x}\kappa(s)\ee^{\int_{0}^{s}\lambda(l)\dl}\ds}, \forall  x \in \mathbb{R}.
\end{equation*}
By using \eqref{3.3} we know that
\begin{equation}\label{3.9}
1+u_0\int_{0}^{x}\kappa(s)\ee^{\int_{0}^{s}\lambda(l)\dl}\ds>0, \forall x \in \R.
\end{equation}
Therefore, by definition of $V(x)$,  we have that $V(x)>0, \forall x\in\R$. On the other hand, by using \eqref{3.8}, we have that
\begin{align}\label{3.10}
V'(x)&=\frac{1}{c-\chi P'(x)}V(x) \left( \left( 1+\frac{\chi}{\sigma^{2}} P(x)\right)-\left(1+ \frac{\chi}{\sigma^{2}} \right)V(x) \right), \; \forall x \in \mathbb{R},
\end{align}
and
\begin{equation*}
P(x)=\frac{1}{2 \sigma} \int_{\mathbb{R}}\ee^{-\frac{|x-y|}{\sigma}}U(y) \mathrm{d} y=\frac{1}{2 \sigma} \int_{\mathbb{R}} \ee^{-\frac{|y|}{\sigma}} U(x-y) \mathrm{d}y, \; \forall x \in \mathbb{R}. 
\end{equation*}
Since $0\leq U(x)\leq1,$ $\forall x\in\R,$ we have
\begin{equation}\label{3.11}
0\leq P(x)\leq1, \forall x\in\R.
\end{equation}
{
We deduce that for $x \in \mathbb{R}$
\begin{align}\label{3.14}
V'(x)&\leq\frac{1+\frac{\chi}{\sigma^{2}}}{c-\chi P'(x)}V(x) \big( 1-V(x) \big), \; \forall x \in \mathbb{R},
\end{align}
By \eqref{3.14} and the comparison principle, we have that
\begin{equation} \label{3.15}
	0\leq V(x)\leq1, \; \forall x\in\R.
\end{equation}
}
\textbf{Step 3.} Let us prove that
$$
    {0< V'(x)\leq C_U}, \; \forall x \in \mathbb{R}.
$$
Since
\begin{equation*}
0\leq U'(x)\leq C_U,  \; \forall x \in \mathbb{R},
\end{equation*}
then we have
\begin{equation}\label{3.16}
		0 \leq P'(x)=\frac{1}{2 \sigma} \int_{\mathbb{R}} e^{-\frac{|y|}{\sigma}}U'(x-y) \dy \leq  \sup_{x \in \R}U'(x)\leq C_U, \; \forall x \in \mathbb{R}. 
	\end{equation}
By \eqref{3.16}, we have that
$$
c-\chi P'(x)\geq c-\chi C_U, \; \forall x \in \mathbb{R}. 
$$
Therefore we deduce
\begin{equation}\label{3.17}
    0<\frac{1}{c-\chi P'(x)}\leq\frac{1}{c-\chi C_U}, \; \forall x \in \mathbb{R}. 
\end{equation}
Since $0\leq U(x)\leq 1$, and 
\begin{equation*}
P(x)= \frac{1}{2 \sigma} \int_{\mathbb{R}} \ee^{-\frac{|y|}{\sigma}} U(x-y) \dy,\; \forall x \in \R, 
\end{equation*}
we deduce that 
\begin{equation}\label{3.18}
	0\leq P(x) \leq 1, \; \forall x \in \R. 
\end{equation}
Therefore, by using the fact that $0\leq V(x)\leq1$, $\forall x\in\R$,  \eqref{3.10}, \eqref{3.17}, and \eqref{3.18}, {we have $V(x)(1-V(x))\leq 1/4$ and} we deduce that for $x\in\R$
\begin{equation*}
    V'(x) \leq{\frac{1+\frac{\chi}{\sigma^2}}{4(c-\chi C_U)}} .
\end{equation*}
Using the definition of $C_U$ and $c^2\geq\chi\left(1+\frac{\chi}{\sigma^2}\right)$, we have
\begin{align*}
    V'(x)&\leq \frac{1+\frac{\chi}{\sigma^2}}{4(c-\chi C_U)}  = \frac{1+\frac{\chi}{\sigma^2}}{2\left(c-\sqrt{c^2-\chi\left(1+\frac{\chi}{\sigma^2}\right)}\right)} \\ 
    &=\frac{1+\frac{\chi}{\sigma^2}}{2\left(c-\sqrt{c^2-\chi\left(1+\frac{\chi}{\sigma^2}\right)}\right)}\times \dfrac{c+\sqrt{c^2-\chi\left(1+\frac{\chi}{\sigma^2}\right)}}{c+\sqrt{c^2-\chi\left(1+\frac{\chi}{\sigma^2}\right)}} \\
    &=\frac{1+\frac{\chi}{\sigma^2}}{2\left(c^2-c^2+\chi\left(1+\frac{\chi}{\sigma^2}\right)\right)}\left(c+\sqrt{c^2-\chi\left(1+\frac{\chi}{\sigma^2}\right)}\right) \\ 
    & = C_U
\end{align*}
and we finally reach
\begin{equation}\label{3.19}
    V'(x) \leq C_U .
\end{equation}

On the other hand, by using \eqref{3.7}, we have 
\begin{equation}\label{3.22}
\begin{aligned}
V'(x)&=\dfrac{\lambda(x) u_0\ee^{  \int_{0}^{x}\lambda(s)\ds}\left(\displaystyle 1+u_0 \int_{0}^{x}\kappa(s)\ee^{  \int_{0}^{s}\lambda(l)\dl}\ds \right)-\kappa(x)\left(u_{0}\ee^{ \int_{0}^{x}\lambda(s)\ds} \right)^2 }{\displaystyle\left( 1+u_0  \int_{0}^{x}\kappa(s)\ee^{ \int_{0}^{s}\lambda(l)\dl}\ds \right)^2}\\
&=\dfrac{\displaystyle\lambda(x)u_0\ee^{ \int_{0}^{x}\lambda(s)\ds} +\Lambda (x)}{\displaystyle\left( 1+u_0  \int_{0}^{x}\kappa(s)\ee^{  \int_{0}^{s}\lambda(l)\dl}\ds \right)^2},
\end{aligned}
\end{equation}
where
$$
\Lambda (x):=u_{0}^{2}\ee^{ \int_{0}^{x}\lambda(s)\ds} \left[ \lambda(x) \int_{0}^{x}\kappa(s)\ee^{  \int_{0}^{s}\lambda(l)\dl}\ds-\kappa(x)\ee^{ \int_{0}^{x}\lambda(s)\ds}\right]. 
$$
From \eqref{3.22}, to prove $V'(x)>0$ for $x\in\R$, we only need to prove that
$$
\Lambda _1(x):=\lambda(x)u_0\ee^{\int_{0}^{x}\lambda(s)\ds} +\Lambda (x)>0.
$$
Indeed, we have 
\begin{equation}\label{3.23}
\begin{aligned}
\Lambda (x)=\kappa(x) \left( u_{0}\ee^{ \int_{0}^{x}\lambda(s)\ds}  \right)^2 \left(\frac{\lambda(x)   \int_{0}^{x}\kappa(s)\ee^{  -\int_{s}^{x}\lambda(l)\dl}\ds}{\kappa(x)}-1\right).
\end{aligned}
\end{equation}
By using the definitions $\lambda(x)$ and $\kappa(x)$ in \eqref{3.4} and \eqref{3.5}, and the formula \eqref{3.23}, we deduce that
\begin{equation}\label{3.24}
    \Lambda (x)=\kappa(x) \left( u_{0}\ee^{ \int_{0}^{x}\lambda(s)\ds}  \right)^2 \left(  \int_{0}^{x}\frac{1+\frac{\chi}{\sigma^2}P(x)}{c-\chi P'(s)}\ee^{\mathop{\textstyle-   \int_{s}^{x}\frac{1+\frac{\chi}{\sigma^{2}}P(l) }{c-\chi P'(l)}\dl}}\ds-1\right).
\end{equation}
Since by assumption $U$ is increasing, it follows from \eqref{3.16} that $P$ is increasing. Then, for any $s<x$, we have $P(s)\leq P(x)$. We deduce that
\begin{equation*}
\begin{aligned}
    \int_{0}^{x}\frac{1+\frac{\chi}{\sigma^{2}}P(x)}{c-\chi P'(s)}\ee^{\mathop{\textstyle-   \int_{s}^{x}\frac{1+\frac{\chi}{\sigma^{2}}P(l)}{c-\chi P'(l)}\dl}}\ds& \geq \int_{0}^{x}\frac{1+\frac{\chi}{\sigma^{2}}P(s)}{c-\chi P'(s)}\ee^{\mathop{\textstyle-   \int_{s}^{x}\frac{1+\frac{\chi}{\sigma^{2}}P(l)}{c-\chi P'(l)}\dl}}\ds\\
 &=\int_{0}^{x}\dfrac{\d}{\ds}\left(\ee^{-   \int_{s}^{x}\frac{1+\frac{\chi}{\sigma^{2}}P(l)}{c-\chi P'(l)}\dl}\right)\ds\\
    &=1-\ee^{\mathop{\textstyle-   \int_{0}^{x}\frac{1+\frac{\chi}{\sigma^{2}}P(l)}{c-\chi P'(l)}\dl}}\\
 &=1- \ee^{ - \int_{0}^{x}\lambda(s)\ds},
\end{aligned}
\end{equation*}
therefore by combining  \eqref{3.24} and the above inequality, we obtain 
\begin{equation}\label{3.25}
\begin{aligned}
\Lambda (x)&\geq \kappa(x)\left( u_{0}\ee^{ \int_{0}^{x}\lambda(s)\ds}  \right)^2 \left(1- \ee^{ - \int_{0}^{x}\lambda(s)\ds}-1\right)\\
&=-\kappa(x)u_{0}^{2}\ee^{   \int_{0}^{x}\lambda(s)\ds}.
\end{aligned}
\end{equation}
By using \eqref{3.25} and the definitions of $\lambda(x)$ and $\kappa(x)$ for $x\in\R$, we have that
\begin{equation}\label{3.26}
\begin{aligned}
\Lambda _1(x)&=\lambda(x)u_0\ee^{   \int_{0}^{x}\lambda(s)\ds} +\Lambda (x)\\
&\geq \lambda(x)u_0\ee^{   \int_{0}^{x}\lambda(s)\ds}-\kappa(x)u_{0}^{2}\ee^{  \int_{0}^{x}\lambda(s)\ds}  \\
   & =u_0\ee^{  \int_{0}^{x}\lambda(s)\ds}\left[\frac{1+\frac{\chi}{\sigma^{2}}P(x) }{c-\chi P'(x)}-u_0 \frac{1+\frac{\chi}{\sigma^{2}} }{c-\chi P'(x)}\right]\\
   &=\frac{u_0\ee^{  \int_{0}^{x}\lambda(s)\ds}}{c-\chi P'(x)}\left[1+\frac{\chi}{\sigma^{2}}P(x) -u_0 \left(1+\frac{\chi}{\sigma^{2}} \right)\right].
\end{aligned}
\end{equation}
To conclude it remains to recall that by assumption we have 
$$
u_0<{\frac{\sigma^2}{2 \left( \sigma^2+\chi\right)}\left(1-\sqrt{1-\frac{\chi}{c^2}\left(1+\frac{\chi}{\sigma^2}\right)}\right)} < \frac{\sigma^2}{\sigma^2+\chi},
$$
therefore since $P(x)\geq 0$ and $c-\chi P'(x)>0$, we have that
\begin{equation*}
\Lambda _1(x)\geq \left(u_0\ee^{  \int_{0}^{x}\lambda(s)\ds} \right)\frac{1+\frac{\chi}{\sigma^{2} }}{c-\chi P'(x)}\left[ \frac{\sigma^2}{\sigma^2+\chi}-u_0 \right] >0,  \; \forall x \in \R,
\end{equation*}
which implies 
\begin{equation} \label{3.27}
V'(x)>0, \; \forall x \in \R.
\end{equation}
The conclusion of the Step 3 now follows from \eqref{3.19} and  \eqref{3.27}. The proof is completed.
\end{proof}

\medskip 
Let $\eta>0$. Let $\BUC\left( \R \right)$ be the space of bounded and uniformly continuous maps from $\R$ to itself. Define the weighted space of continuous functions 
$$
BUC_\eta  \left( \R\right)=\left\{ U \in C\left( \R\right): x \to e^{-\eta \vert x \vert } U(x) \in \BUC\left( \R\right) \right\}, 
$$
and  the weighted space  $n$-times continuous differentiable functions
$$
BUC_\eta^n \left( \R\right)=\left\{ U \in C^n \left( \R\right):  x \to e^{-\eta \vert x \vert } U^{(k)}(x) \in \BUC\left( \R\right), \forall k=0,\ldots,n \right\},
$$
which is a Banach space endowed with the norm  
\begin{equation}\label{3.28}
	\|U\|_{n,\eta}:=\sup_{x\in\R}\ee^{-\eta |x|}|U(x)|+\sup_{x\in\R}\ee^{-\eta |x|}|U'(x)|+\ldots+\sup_{x\in\R}\ee^{-\eta |x|}|U^{(n)}(x)|. 
\end{equation}
Here we will use the above weighted space of $\BUC^1_\eta \left( \R \right)$ maps to ensure that   
	$$
\mathcal{A} = \left\{ U \in C^1\left( \R\right):	0\leq U(x) \leq 1, \text{ and } 0\leq U'(x)\leq C_U, \forall x \in \mathbb{R} \right\}, 
	$$
is a closed subset of $\BUC^1_\eta \left( \R \right)$. As a consequence, the subset $\mathcal{A}$ equipped with the distance
$$
 d_\eta (U_1,U_2)=\Vert U_1-U_2\Vert_{1,\eta}.
$$
is a complete metric space.

\begin{lemma}[Compactness of $\mathcal{T}$]\label{LEM3.6}
Let Assumption \ref{ASS3.2} be satisfied. Then the set $\overline{ \mathcal{T} \left(\mathcal{A} \right)}$
is a compact subset of the metric space $\mathcal{A}$ equipped with the distance $ d_\eta $. 
\end{lemma}
\begin{proof}
Let $ \left\{ U_n \right\}_{n \geq 0} \subset \mathcal{A}$ be a sequence, and define the corresponding sequence $ \left\{P_n\right\}_{n \geq 0}$ solution of equation \eqref{3.6} where $U$ is replaced by $U_n$. Define the corresponding  sequences $\left\{ \lambda_n \right\}_{n \geq 0} $ and $\left\{ \kappa_n \right\}_{n \geq 0} $ by using \eqref{3.4} and \eqref{3.5} where $P(x)$ is replaced by $P_n(x)$. Denote $V_n=\mathcal{T}(U_n), \forall  n\in\mathbb{N}$.
By Lemma \ref{LEM3.3}, we know that $\mathcal{T}(\mathcal{A})\subset\mathcal{A}$. Therefore we have
\begin{equation} \label{3.29}
0< V_n(x)\leq1 \text{ and }0< V_n'(x)\leq C_U, \; \forall x\in\R.
\end{equation}
Similarly to equation \eqref{3.8} in the proof of Lemma \ref{LEM3.3}, we have that
\begin{equation}  \label{3.30}
	V_n'(x)=\lambda_n(x)V_n(x)- \kappa_n(x)V_n^2(x), \; \forall x\in\R,
\end{equation}
where
$$
\lambda_n(x)=\frac{1+\frac{\chi}{\sigma^{2}}P_n(x) }{c-\chi P_n'(x)} , \; \forall x\in\R,
$$
and
$$\quad\kappa_n(x)=\frac{1+ \frac{\chi}{\sigma^{2}}}{c-\chi P_n'(x)} , \; \forall x\in\R,$$
and $P_n(x)$ is the unique solution of the elliptic equation
\begin{equation*} 
		P_n( x)-\sigma^{2} P_n''( x)=U_n(x), \forall x \in \mathbb{R}.
	\end{equation*}
It follows from Lemma \ref{LEM2.4} that the map $x \to P_n(x)$ solving the above equation is an increasing $C^3$ function.
By \eqref{3.16}, we have that $c-\chi P_n'(x)>0$. Since $U_n(x)\in C^1(\R)$, we have that $\lambda_n(x), \kappa_n(x)\in C^2(\R)$, and $V_n'(x)\in C^1(\R)$. Then, we obtain that $V_n(x)\in C^2(\R).$ Therefore by using \eqref{3.29} and \eqref{3.30}, we deduce that the families $V_n'|_{[-k, k]}$ and $V_n''|_{[-k, k]}$ are uniformly Lipschitz continuous on $[-k,k]$ for each $k\in\mathbb{N}$. Applying Ascoli-Arzel\`{a} theorem, we have that the sets $\{V_n|_{[-k, k]}\}_{n>0}$, and $\{V_n'|_{[-k, k]}\}_{n>0}$ are relatively compact on $[-k, k]$ for each $k\in\mathbb{N}$.

Using a diagonal extraction process, there exists a sub-sequence $n_p$ and a bounded continuous function $V$ such that $V_{n_p}\rightarrow V$ uniformly on every compact subset of $\R$ as $p\rightarrow\infty$. Indeed, recall that $0< V_n(x)\leq1$ and $0< V_n'(x)\leq C_U$, $ \forall x\in\R$. By the Ascoli-Arzel\`{a} theorem, there exists a sub-sequence $\{V_{m_{p}^1}\}_{p\geq0}$ of $V_{n}$ and a function $V_1 \in C^1([-1,1])$ such that
$$
\lim_{p \to \infty } \Vert V_{m_{p}^1}- V_1 \Vert_{C^1([-1,1])}=0.
$$
Now we can extract $\{V_{m_{p}^{2}}\}_{p\geq0}$ a sub-sequence of $\{V_{m_{p}^1}\}_{p\geq0}$, and function $V_2 \in C^1([-2,2])$
$$
\lim_{p \to \infty } \Vert V_{m_{p}^2}- V_2 \Vert_{C^1([-2,2])}=0.
$$
By construction, we will have
$$
V_2(x)=V_1(x), \forall x \in [-1,1].
$$
Replacing eventually $m_{1}^{2}$ by $m_{1}^{1}$, we can assume that $m_{p}^{2}=\{m_{1}^{1}, m_{2}^{2},\cdots,m_{p}^{2},\cdots\}$.  Proceeding by induction, we can find a $\{V_{m_{p}^{k}}\}_{p\geq0}$  a sub-sequence $\{V_{m_{p}^{k-1}}\}_{p\geq0}$, such that
$$
m_{p}^{k}=m_{p}^{k-1}, \forall p=1,\ldots,k-1,
$$
and a function  $V_k \in C^1([-k,k])$ such that
$$
\lim_{p \to \infty } \Vert V_{m_{p}^k}- V_k \Vert_{C^1([-k,k])}=0.
$$
Set
$$
n_p=m_{p}^{p},
$$
and
$$V(x)=V_{n}(x), \text{ for any }x\in[-k, k], \forall k\geq1.$$
Then the sub-sequence $\{V_{n_{p}}\}_{p\geq0}$ converges locally uniformly with respect to the $C^1$-norm, so we can define
$$
V(x)= \lim_{p \to \infty} V_{n_{p}}(x), \forall x \in \R,
$$
and
$$
V'(x)= \lim_{p \to \infty} V_{n_{p}}'(x), \forall x \in \R.
$$
By construction, we will have
\begin{equation}\label{3.31}
0< V(x)\leq1, \text{ and } 0< V'(x)\leq C_U, \forall x \in \R.
\end{equation}
Now, we are ready to show that $\|V_{n_p}-V\|_{1,\eta}\rightarrow 0$ as $p\rightarrow +\infty$.
Let $\varepsilon>0$ be given. Let $k$ be large enough to satisfy 
\begin{equation}\label{3.32}
\ee^{-\eta k}\leq\frac{\varepsilon}{4}\min\left\{1, \frac{1}{C_U}\right\}.
\end{equation}
For all $k$ large enough, since $0< V_{n_p}(x)\leq1$, $0< V_{n_p}'(x)\leq C_U$, $\forall x\in\R$ and \eqref{3.31}, we deduce that  
\begin{equation}\label{3.33}
\begin{aligned}
	\sup_{x\in\R\setminus[-k, k]}\ee^{-\eta |x|}|V_{n_p}(x)-V(x)|&\leq\ee^{-\eta k}\sup_{x\in\R}(|V_{n_p}(x)|+|V(x)|) \\
  &\leq 2\ee^{-\eta k}\\
  &\leq \frac{\varepsilon}{2},
\end{aligned}
\end{equation}
and
\begin{equation}\label{3.34}
\begin{aligned}
\sup_{x\in\R\setminus[-k, k]}\ee^{-\eta
|x|}|V_{n_p}'(x)-V'(x)| & \leq \ee^{-\eta k}\sup_{x\in\R}(|V_{n_p}'(x)|+|V'(x)|) \\
  &\leq 2C_U\ee^{-\eta k}\\
  &\leq \frac{\varepsilon}{2}.
\end{aligned}
\end{equation}
Moreover, since $V_{n_p}$ converges locally uniformly to $V$ and $V_{n_p}'$ converges locally uniformly to $V'$, for any fixed $x\in[-k, k]$, there exists an integer $p_0>0$ such that
\begin{equation}\label{3.35}
 \sup_{x\in[-k, k]}\ee^{-\eta |x|}|V_{n_p}(x)-V(x)|\leq\frac{\varepsilon}{2}, \forall p\geq p_0,
\end{equation}
and
\begin{equation}\label{3.36}
\sup_{x\in[-k, k]}\ee^{-\eta |x|}|V_{n_p}'(x)-V'(x)|\leq\frac{\varepsilon}{2}, \forall p\geq p_0.
\end{equation}
It follows from \eqref{3.33}, \eqref{3.34}, \eqref{3.35} and \eqref{3.36} that for $p \geq p_0$, 
\begin{equation*}
\begin{aligned}
\|V_{n_p}-V\|_{1,\eta}=&\Bigg\{\max\left\{\sup_{x\in\R\setminus[-k, k]}\ee^{-\eta |x|}|V_{n_p}(x)-V(x)|,\sup_{x\in[-k, k]}\ee^{-\eta
|x|}|V_{n_p}(x)-V(x)|\right\}\\
&+\max\left\{\sup_{x\in\R\setminus[-k, k]}\ee^{-\eta |x|}|V_{n_p}'(x)-V'(x)|,\sup_{x\in[-k, k]}\ee^{-\eta
|x|}|V_{n_p}'(x)-V'(x)|\right\}\Bigg\} \\
\leq &  \frac{\varepsilon}{2}+ \frac{\varepsilon}{2}= \varepsilon. 
\end{aligned}
\end{equation*}
Since the above inequality is true for any $\varepsilon>0$, this completes the proof of lemma.
\end{proof}

\vspace{0.5cm}

\medskip 

The most difficult part of the proof of existence of traveling waves is the continuity of the map  $\mathcal{T}: \mathcal{A} \to \mathcal{A}$.  To consider this problem,  we decompose the real line into several intervals $(-\infty, -K]$, $[-K,K]$, and $[K,\infty)$. {Before proving the continuity of $\mathcal T$, we establish the continuity of its components separately. 
\begin{lemma}[Continuity of $P$, $P'$, $\lambda$ and $\kappa$].\label{LEM3.5}
	Let Assumption \ref{ASS3.2} be satisfied.  Assume that $0<\eta<\frac{1}{\sigma}$.  Let $U_1, U_2\in \mathcal{A}$ and define, for $i=1,2,$
	\begin{align*}
		P_i(x) &=\frac{1}{2\sigma} \int_{\mathbb R} e^{-\frac{|x-y|}{\sigma}} U_i(y)\dd y,  &
		 P_i'(x)&=\frac{1}{2\sigma^2} \int_{\mathbb R}-\mathrm{sign}(x-y) e^{-\frac{|x-y|}{\sigma}} U_i(y)\dd y, \\
		 \lambda_i(x)&= \dfrac{1+\frac{\chi}{\sigma^2}P_i(x) }{c-\chi P_i'(x)}, &
		 \kappa_i(x) &= \dfrac{1+\frac{\chi}{\sigma^2}}{c-\chi P_i'(x)} .
	\end{align*}
	There exist continuous functions of $x\in\mathbb R$, $C_P(x)$, $C_\lambda(x)$ and $C_\kappa(x)$ such that, for all $x\in\mathbb R$, 
	\begin{align}
		\label{B.37}	|P_1(x)-P_2(x)| &\leq C_P(x)\Vert U_1-U_2\Vert_{0, \eta},  \\ 
		\label{B.38}	|P_1'(x)-P_2'(x)| &\leq \frac{1}{\sigma}C_P(x)\Vert U_1-U_2\Vert_{0, \eta},  \\ 
		\label{B.39}	|\lambda_1(x)-\lambda_2(x)| &\leq C_\lambda(x)\Vert U_1-U_2\Vert_{0, \eta},  \\ 
		\label{B.40}	|\kappa_1(x)-\kappa_2(x)| &\leq C_\kappa(x)\Vert U_1-U_2\Vert_{0, \eta}.
	\end{align}
	{The functions $C_P(x)$, $C_\lambda(x)$ and $C_\kappa(x)$ do not depend on the particular choice of $U_1\in\mathcal{A}$ and $U_2\in\mathcal{A}$ but only on $\eta$, $\sigma$ and $\chi$.}
\end{lemma}
\begin{proof}
	\noindent\textbf{Step 1:} We show \eqref{B.37}. 
	We have,  for $x>0$:
	\begin{align}
		\nonumber|P_1(x)-P_{2}(x)|&=  \frac{1}{2 \sigma}\left|\int_{-\infty}^{+\infty} \ee^{-\frac{|x-y|}{\sigma}}(U_1(y)-U_{2}(y)) \mathrm{d} y\right| \\
		\nonumber& =\frac{1}{2 \sigma}\left|\int_{-\infty}^{+\infty} \ee^{-\frac{|x-y|}{\sigma}+\eta|y|}\ee^{-\eta |y|}(U_1(y)-U_{2}(y)) \mathrm{d} y\right| \\
		\nonumber&\leq \frac{1}{2\sigma}\int_{\mathbb R} \ee^{ -\frac{|x-y|}{\sigma}+\eta|y|}\dd y\Vert U_1-U_2\Vert_{0, \eta} \\
		\nonumber&=\frac{1}{2\sigma}\Vert U_1-U_2\Vert_{0, \eta} \left(
		\int_{-\infty}^{0} e^{-\frac{x-y}{\sigma}-\eta y}\dd y
		+\int_{0}^x e^{-\frac{x-y}{\sigma}+\eta y} \dd y
		+\int_{x}^{+\infty} \ee^{\frac{x-y}{\sigma}+\eta y}\dd y\right)\\
		\label{A.48}&=\frac{1}{2\sigma}\Vert U_1-U_2\Vert_{0, \eta} \left(\dfrac{1}{\frac{1}{\sigma}-\eta}e^{-\frac{x}{\sigma}} + \dfrac{1}{\frac{1}{\sigma}+\eta}\left(\ee^{\eta x}-e^{-\frac{x}{\sigma}}\right) + e^{\eta x}\dfrac{1}{\frac{1}{\sigma}-\eta}\right),
	\end{align}
	and similarly for $x<0$: 
	\begin{equation}\label{A.49}
		|P_1(x)-P_{2}(x)| \leq \frac{1}{2\sigma}\Vert U_1-U_2\Vert_{0, \eta} \left(\dfrac{1}{\frac{1}{\sigma}-\eta}e^{-\frac{|x|}{\sigma}} + \dfrac{1}{\frac{1}{\sigma}+\eta}\left(\ee^{\eta|x|}-e^{-\frac{1}{\sigma}|x|}\right) + e^{\eta |x|}\dfrac{1}{\frac{1}{\sigma}-\eta}\right).
	\end{equation}
	Rearranging the terms in \eqref{A.48} and \eqref{A.49}  we have 
	\begin{equation}\label{3.56}
		\begin{aligned}
			|P_1(x)-P_2(x)|&\leq \frac{1}{2 \sigma}\left[\left(\frac{1}{\sigma}-\eta\right)^{-1}\left(\ee^{-\frac{|x|}{\sigma}}+\ee^{\eta |x|}\right)+\left(\frac{1}{\sigma}+\eta\right)^{-1}\left(\ee^{\eta |x|}-\ee^{-\frac{|x|}{\sigma}}\right)\right] \|U_1-U_{2}\|_{0,\eta} 
		\end{aligned}
	\end{equation}
	and \eqref{B.37} is proved.\medskip
		
	\noindent\textbf{Step 2:} We show \eqref{B.38}. We have 
	\begin{align}
		\nonumber|P_1'(x)-P_{2}'(x)|&=  \frac{1}{2 \sigma^2}\left|\int_{-\infty}^{+\infty}-\mathrm{sign}(x-y)\ee^{-\frac{|x-y|}{\sigma}}(U_1(y)-U_{2}(y)) \mathrm{d} y\right|\\
		\nonumber& \leq\frac{1}{2 \sigma^2}\int_{-\infty}^{+\infty} \ee^{-\frac{|x-y|}{\sigma}+\eta|y|}\ee^{-\eta |y|}\left|U_1(y)-U_{2}(y)\right| \mathrm{d} y,  
	\end{align}
	so that the exact computations leading to  \eqref{3.56} can be reproduced, and we have
	\begin{equation*}
		|P_1'(x)-P_2'(x)|\leq \frac{1}{\sigma}C_P(x)\Vert U_1-U_2\Vert_{0,  \eta}.
	\end{equation*}
	\eqref{B.38} is proved. \medskip 

	\noindent\textbf{Step 3:} We show \eqref{B.39}.
		It follows from the definitions of  $\lambda_1(x)$ and $\lambda_2(x)$ that, for all $x\in\mathbb R$, 
		\begin{align}
			\nonumber	&|\lambda_2(x)-\lambda_{1}(x)|\\
			\nonumber	&=\left|\frac{1+ \frac{\chi}{\sigma^{2}}P_2(x)}{c-\chi P_2'(x)}-\frac{1+ \frac{\chi}{\sigma^{2}}P_{1}(x)}{c-\chi P_{1}'(x)}\right|\\
			\label{B.44}	&=\frac{\chi}{|c-\chi P_2'(x)||c-\chi P_{1}'(x)|}\left|\frac{c}{\sigma^2}(P_2(x)-P_{1}(x))+P_2'(x)-P_{1}'(x)+\frac{\chi}{\sigma^2}(P_{1}(x)P_2'(x)-P_2(x)P_{1}'(x))\right|.
		\end{align}
		Since by Definition \ref{DEF3.1} we have $U_i'(x)\leq C_U$ ($i=1,2$), then $P_i'(x)=\int_{\mathbb R} \frac{1}{2\sigma}e^{-\frac{|x-y|}{\sigma}} U'(y)\dd y\leq C_U$ ($i=1,2$), therefore
		\begin{equation}\label{B.45}
		    c-\chi P_i'(x) \geq c-\chi C_U=\frac{1}{2}\left(c-\sqrt{c^2-\chi\left(1+\frac{\chi}{\sigma^2}\right)}\right) >0, \qquad i=1,2.
		\end{equation}
		It follows from \eqref{B.44} and \eqref{B.45} that
		\begin{equation*}
			\begin{aligned}
			    |\lambda_2(x)-\lambda_{1}(x)|&\leq\frac{c\chi}{\sigma^2(c-\chi C_U)^2}|P_2(x)-P_{1}(x)|+\frac{\chi}{(c-\chi C_U)^2}|P_2'(x)-P_{1}'(x)| \\ 
			    &\quad +\frac{\chi}{\sigma^2(c-\chi C_U)^2}|P_{1}(x)P_2'(x)-P_2(x)P_{1}'(x)|\\
				&\leq\frac{c\chi}{\sigma^2(c-\chi C_U)^2}|P_2(x)-P_{1}(x)|+\frac{\chi}{(c-\chi C_U)^2}|P_2'(x)-P_{1}'(x)| \\
				&\quad+\frac{\chi}{\sigma^2(c-\chi C_U)^2}\left(P_1(x)|P_2'(x)-P_{1}'(x)|+|P'_1(x)||P_{1}(x)-P_2(x)|\right).
			\end{aligned}
		\end{equation*}
		Using the fact that $0\leq P_1(x)\leq1$ and $0\leq P'_1(x)\leq C_U$ for $x\in\R$, then \eqref{B.39} is a consequence of \eqref{B.37} and \eqref{B.38}.\medskip

	\noindent\textbf{Step 4:} We show \eqref{B.40}. We have:
	\begin{align*}
	\dfrac{1}{1+ \frac{\chi}{ \sigma^2}}	|\kappa_1(x)-\kappa_2(x)|&=\left|\frac{1}{c-\chi P_1'(x)}-\frac{1}{c-\chi P_{2}'(x)}\right|=\left|\dfrac{c-\chi P_1'(x)-c+\chi P_2'(x)}{\big(c-\chi P_2'(x)\big)\big(c-P_1'(x)\big)}\right|\\ 
		&=\chi\left|\dfrac{P_2'(x)-P_1'(x)}{\big(c-\chi P_2'(x)\big)\big(c-P_1'(x)\big)}\right| \leq \dfrac{\chi}{(c-\chi C_U)^2}|P_2'(x)-P_1'(x)|.
	\end{align*}
	Thus \eqref{B.40} is a consequence of \eqref{B.38}. Lemma \ref{LEM3.5} is proved.
\end{proof}
}

\begin{lemma}[Continuity of $\mathcal{T}$]\label{LEM3.7}
	Let Assumption \ref{ASS3.2} be satisfied. Assume that $0<\eta<\frac{1}{\sigma}$. Then the map $\mathcal{T}: \mathcal{A} \to \mathcal{A}$ is continuous on $\mathcal{A}$ endowed with distance $d(U_1,U_2)=\Vert U_1-U_2\Vert_{1,\eta}$.
\end{lemma}
\begin{proof} Let $U_0 \in \mathcal{A}$ be fixed, and $U \in \mathcal{A}$, and define
	$$
	V_0= \mathcal{T}(U_0) \text{ and }V= \mathcal{T}(U).
	$$  
	
	\medskip 		
	\noindent \textbf{Part A:} We prove that 
	{for each admissible profile  $U_0 \in \mathcal A$ and  $\varepsilon>0$, there is a $\delta_1>0$ such that 
	\begin{equation}\label{B.46}
		\|V-V_0\|_{0,\eta} \leq \frac{\varepsilon}{2},
	\end{equation}
	whenever 
	$$
	\Vert U-U_0\Vert_{0,\eta} \leq \delta_1.
	$$ 
	\medskip

	Let $K>0$ be such that 
	\begin{equation*}
		e^{-\eta K} \leq \frac{\varepsilon}{12}.
	\end{equation*}
	Then since $V\in\mathcal{A}$ by Lemma \ref{LEM3.3}, we have $0\leq V(x)\leq 1$ and $0\leq V_0(x)\leq 1$ for all $x\in\mathbb R$, therefore 
	\begin{align*}
		\Vert V-V_0\Vert_{0, \eta} & = \sup_{x\in\mathbb R} e^{-\eta \vert x \vert }|V(x)-V_0(x)| \\ 
		&\leq \sup_{x\leq -K} e^{-\eta |x|}|V(x)-V_0(x)| + \sup_{|x|\leq K} e^{-\eta |x|}|V(x)-V_0(x)|+ \sup_{x\geq K}e^{-\eta |x|}|V(x)-V_0(x)|\\ 
		&\leq e^{-\eta K}\left(\sup_{x\leq -K}|V(x)-V_0(x)|+\sup_{x\geq K}|V(x)-V_0(x)|\right) + \sup_{|x|\leq K}e^{-\eta |x|}|V(x)-V_0(x)| \\ 
		&\leq 4 e^{-\eta K}+ \sup_{|x|\leq K}e^{-\eta |x|}|V(x)-V_0(x)| \leq \frac{2\varepsilon}{6}+\sup_{|x|\leq K}e^{-\eta |x|}|V(x)-V_0(x)|.
	\end{align*}
	Thus  there remains only to  establish that
	\begin{equation}\label{A.46}
		\sup_{|x|\leq K}e^{-\eta |x|}|V(x)-V_0(x)|\leq \frac{\varepsilon}{6},
	\end{equation}
	if $\Vert U-U_0\Vert_{1, \eta}\leq \delta_1$, for $\delta_1>0$ sufficiently small.
	}
	Recall that
	\begin{equation*}
		V(x)=\dfrac{\displaystyle u_0\exp\left(\int_{0}^{x}\lambda(s)\ds\right)}{\displaystyle 1+u_0\int_{0}^{x}\kappa(s)\exp\left(\int_{0}^{s}\lambda(l)\dl\right)\ds},
	\end{equation*}
	wherein
	\begin{equation*}
		\lambda(x)=\frac{1+\frac{\chi}{\sigma^{2}}P(x) }{c-\chi P'(x)},
	\end{equation*}
	and
	\begin{equation*}
		\quad\kappa(x)=\frac{1+ \frac{\chi}{\sigma^{2}}}{c-\chi P'(x)},
	\end{equation*}
	and $P(x)$ is the unique solution of the elliptic equation
	\begin{equation*}
		P(x)-\sigma^{2} P''(x)=U(x), \forall x \in \mathbb{R}.
	\end{equation*}
	By using the definitions of $V_0(x)$ and $V(x)$ for $x\in\R$, we find that
	\begin{align}
		\nonumber	|V(x)-V_0(x)|&=\left|\frac{u_0\exp\left(\displaystyle\int_{0}^{x}\lambda(s)\ds\right)}{1+u_0\displaystyle\int_{0}^{x}\kappa(s)\exp\left(\displaystyle\int_{0}^{s}\lambda(l)\dl\right)\ds}
		\nonumber	-\frac{u_0\exp\left(\displaystyle\int_{0}^{x}\lambda_0(s)\ds\right)}{1+u_0\displaystyle\int_{0}^{x}\kappa_0(s)\exp\left(\displaystyle\int_{0}^{s}\lambda_0(l)\dl\right)\ds}\right|\\
		\nonumber	&=\frac{u_0}{\left|1+u_0\dis\int_{0}^{x}\kappa(s)\exp\left(\dis\int_{0}^{s}\lambda(l)\dl\right)\ds\right|	\left|1+u_0\dis\int_{0}^{x}\kappa_{0}(s)\exp\left(\dis\int_{0}^{s}\lambda_{0}(l)\dl\right)\ds\right|}\\
		\nonumber	&\quad \times\Bigg|\exp\left(\dis\int_{0}^{x}\lambda(s)\ds\right)-\exp\left(\dis\int_{0}^{x}\lambda_{0}(s)\ds\right)\\
		\nonumber	&\quad \quad+u_{0}\exp\left(\dis\int_{0}^{x}\lambda(s)\ds\right)\dis\int_{0}^{x}\kappa_{0}(s)\exp\left(\dis\int_{0}^{s}\lambda_{0}(l)\dl\right)\ds\\
		\label{3.42}	&\quad \quad-u_{0}\exp\left(\dis\int_{0}^{x}\lambda_{0}(s)\ds\right)\dis\int_{0}^{x}\kappa(s)\exp\left(\dis\int_{0}^{s}\lambda(l)\dl\right)\ds\Bigg|.
	\end{align}
	Since
	\begin{equation*}
		\frac{1}{1+u_0\dis\int_{0}^{x}\kappa_0(s)\exp\left(\dis\int_{0}^{s}\lambda_0(l)\dl\right)\ds}=\frac{V_{0}(x)}{u_0\exp\left(\dis\int_{0}^{x}\lambda_{0}(s)\ds\right)},
	\end{equation*}
	and similarly
	\begin{equation*}
		\frac{1}{1+u_0\dis\int_{0}^{x}\kappa(s)\exp\left(\dis\int_{0}^{s}\lambda(l)\dl\right)\ds}=\frac{V(x)}{u_0\exp\left({\dis\int_{0}^{x}\lambda(s)\ds}\right)},
	\end{equation*}
	we have that
	\begin{align*}
		\nonumber	&|V(x)-V_0(x)|\\
		\nonumber	&=\left|\frac{V(x)V_0(x)}{u_{0}\exp\left({\dis\int_{0}^{x}\lambda(s)\ds}\right)\exp\left({\dis\int_{0}^{x}\lambda_{0}(s)\ds}\right)}\right|
		\nonumber	\Bigg|\exp\left({\dis\int_{0}^{x}\lambda(s)\ds}\right)-\exp\left({\dis\int_{0}^{x}\lambda_{0}(s)\ds}\right) \\
		\nonumber	&\quad +u_{0}\exp\left(\dis\int_{0}^{x}\lambda(s)\ds\right)\dis\int_{0}^{x}\kappa_{0}(s)\exp\left(\dis\int_{0}^{s}\lambda_{0}(l)\dl\right)\ds\\
		\nonumber	&\quad -u_{0}\exp\left(\dis\int_{0}^{x}\lambda_{0}(s)\ds\right)\dis\int_{0}^{x}\kappa(s)\exp\left(\dis\int_{0}^{s}\lambda(l)\dl\right)\ds\Bigg|\\
		\nonumber	&=\dfrac{ \left|V(x)V_0(x)\right|}{u_0}
		\nonumber	\Bigg|\exp\left({\dis-\int_{0}^{x}\lambda_0(s)\ds}\right)-\exp\left({\dis-\int_{0}^{x}\lambda(s)\ds}\right) \\
		\nonumber	&\quad +u_{0}\exp\left(\dis-\int_{0}^{x}\lambda_0(s)\ds\right)\dis\int_{0}^{x}\kappa_{0}(s)\exp\left(\dis\int_{0}^{s}\lambda_{0}(l)\dl\right)\ds\\
	 	&\quad -u_{0}\exp\left(\dis-\int_{0}^{x}\lambda(s)\ds\right)\dis\int_{0}^{x}\kappa(s)\exp\left(\dis\int_{0}^{s}\lambda(l)\dl\right)\ds\Bigg|, 
	\end{align*}
hence 
\begin{equation}  \label{A.43}
|V(x)-V_0(x)| 	\leq \frac{1}{u_0}H(x) + I(x),
\end{equation}
	where 
	\begin{equation}\label{3.45}
		H(x):=\left|\exp\left(-\dis\int_{0}^{x}\lambda_{0}(s)\ds\right)-\exp\left(-\dis\int_{0}^{x}\lambda(s)\ds\right)\right|,
	\end{equation}
	and
	\begin{equation}\label{3.46}
		I(x):=\left|\dis\int_{0}^{x}\kappa_{0}(s)\exp\left(-\dis\int_{s}^{x}\lambda_{0}(l)\dl\right)\ds
		-\dis\int_{0}^{x}\kappa(s)\exp\left(-\dis\int_{s}^{x}\lambda(l)\dl\right)\ds\right|.
	\end{equation}
	We divide the rest of the proof of Part A into two steps, to estimate $H(x)$ and $I(x)$.
	\medskip

	\noindent\textbf{Step 1:} We show that 
	\begin{equation}\label{A.52}
		H(x) \leq C_H(x)\Vert U-U_0\Vert_{0, \eta}, \; \forall x \in \R,
	\end{equation}
	for some continuous function $C_H(x)$ independent of $\varepsilon$, $U$, $U_0$.
	
	\medskip 
	\noindent By Taylor's theorem, we have that
	\begin{equation}\label{3.47}
		|\ee^{A}-\ee^{B}|\leq|A-B|\ee^{\max\{A, B\}},\quad \forall A, B\in\R.
	\end{equation}
	Now we use \eqref{3.47} to estimate $H(x)$ defined in \eqref{3.45}.
	\begin{equation}\label{3.48}
		\begin{aligned}
			H(x)&\leq \left|\int_{0}^{x}\lambda(s)-\lambda_0(s)\ds\right|\exp\left(\dis\max\left\{-\int_{0}^{x}\lambda(s)\ds,-\int_{0}^{x}\lambda_{0}(s)\ds\right\}\right)\\
			&\leq \exp\left(\dis\max\left\{-\int_{0}^{x}\lambda(s)\ds,-\int_{0}^{x}\lambda_{0}(s)\ds\right\}\right)\int_{0}^{x}\left|\lambda(s)-\lambda_0(s)\right|\ds.
		\end{aligned}
	\end{equation}
	Recall from \eqref{B.39} in Lemma \ref{LEM3.5} that there is a continuous function $C_\lambda(x)$ such that 
	\begin{equation*}
		|\lambda(x)-\lambda_0(x)|\leq C_\lambda(x)\Vert U-U_0\Vert_{0, \eta}, \; \forall x \in \R.
	\end{equation*}
	Thus we can rewrite \eqref{3.48} as
	\begin{equation}\label{A.54}
		H(x) \leq \exp\left(-\int_{0}^{x}\lambda(s)\ds,-\int_{0}^{x}\lambda_{0}(s)\ds\right)\int_0^xC_\lambda(s)\dd s \Vert U-U_0\Vert_{0, \eta}.
	\end{equation}
	Next recall the definition of $\lambda(x)$:
	\begin{equation*}
		\lambda(x)=\dfrac{1+\frac{\chi}{\sigma^2}P(x)}{c-\chi P'(x)}, \; \forall x \in \R.
	\end{equation*}
	Since by Definition \ref{DEF3.1} we have $U'(x)\leq C_U$, then $P'(x)=\int_{\mathbb R} \frac{1}{2\sigma}e^{-\frac{|x-y|}{\sigma}} U'(y)\dd y\leq C_U$, therefore
	\begin{equation*}
	    c-\chi P'(x) \geq c-\chi C_U = \frac{1}{2}\left(c-\sqrt{c^2-\chi\left(1+\frac{\chi}{\sigma^2}\right)}\right)>0,  \; \forall x \in \R,
	\end{equation*}
	and therefore 
	\begin{equation*}
	    \lambda(x)=\dfrac{1+\frac{\chi}{\sigma^2}P(x)}{c-\chi P'(x)}\leq \frac{1+\frac{\chi}{\sigma^2}}{c-\chi C_U}, \; \forall x \in \R. 
	\end{equation*}
	Clearly, we have the same upper bound for $\lambda_0(x)$ and $\lambda(x)$, and \eqref{A.54} becomes
	\begin{equation*}
		H(x)\leq \exp\left(\frac{1+\frac{\chi}{\sigma^2}}{c-\chi C_U}|x|\right) \int_{[0,x]}C_\lambda(s)\dd s\Vert U-U_0\Vert_{0, \eta} = C_H(x)\Vert U-U_0\Vert_{0, \eta},  \; \forall x \in \R,
	\end{equation*}
	where $C_H(x)$ is a continuous function. Therefore, \eqref{A.52} is proved. 
	
	\medskip
	\noindent\textbf{Step 2:} We show that 
	\begin{equation}\label{A.56}
		I(x)\leq C_I(x)\Vert U-U_0\Vert_{0, \eta}, \; \forall x \in \R,
	\end{equation}
	for some continuous function $C_I(x)$ independent from $\varepsilon$, $U$, $U_0$.
	
\medskip 
\noindent 	Indeed we have
	\begin{align}
		\nonumber	I(x)&\leq\left|\dis\int_{0}^{x}\kappa_{0}(s)\exp\left(\dis-\int_{s}^{x}\lambda_{0}(l)\dl\right)\ds
				-\dis\int_{0}^{x}\kappa_{0}(s)\exp\left(\dis-\int_{s}^{x}\lambda(l)\dl\right)\ds\right|\\
		\nonumber	&\quad+\left|\dis\int_{0}^{x}\kappa_{0}(s)\exp\left(\dis-\int_{s}^{x}\lambda(l)\dl\right)\ds
				-\dis\int_{0}^{x}\kappa(s)\exp\left(\dis-\int_{s}^{x}\lambda(l)\dl\right)\ds\right|\\
		\nonumber	&=\left|\dis\int_{0}^{x}\kappa_{0}(s)\left(\exp\left(\dis-\int_{s}^{x}\lambda_{0}(l)\dl\right)
		\nonumber	-\exp\left(\dis-\int_{s}^{x}\lambda(l)\dl\right)\right)\ds\right|\\
		\nonumber 	&\quad+\left|\dis\int_{0}^{x}(\kappa_{0}(s)-\kappa(s))\exp\left(\dis-\int_{s}^{x}\lambda(l)\dl\right)\ds\right| \\ 
		\label{3.63}&=:I_1(x)+I_2(x),
	\end{align}
	where 
	\begin{equation} \label{A.58}
		I_1(x):=\left|\dis\int_{0}^{x}\kappa_{0}(s)\left(\exp\left(\dis-\int_{s}^{x}\lambda_{0}(l)\dl\right) -\exp\left(\dis-\int_{s}^{x}\lambda(l)\dl\right)\right)\ds\right|,
	\end{equation}
	and 
	\begin{equation}\label{A.59}
		I_2(x):= \left|\dis\int_{0}^{x}(\kappa_{0}(s)-\kappa(s))\exp\left(\dis-\int_{s}^{x}\lambda(l)\dl\right)\ds\right|.
	\end{equation}
	Using \eqref{3.47}, we rewrite \eqref{A.58} as 
	\begin{align}
		\label{A.60} I_1(x)\leq\int_0^x |\kappa_0(s)|\exp\left(- \max\left\{-\dis\int_{s}^{x}\lambda(l)\dl,-\int_{s}^{x}\lambda_{0}(l)\dl\right\}\right)\int_s^x |\lambda(l)-\lambda_0(l)|\dd l\dd s.
	\end{align}
	Since by Definition \ref{DEF3.1} we have $U'(x)\leq C_U$, then $P'(x)=\int_{\mathbb R} \frac{1}{2\sigma}e^{-\frac{|x-y|}{\sigma}} U'(y)\dd y\leq C_U$, therefore
	\begin{equation*}
	    c-\chi P'(x) \geq c-\chi C_U = \frac{1}{2}\left(c-\sqrt{c^2-\chi\left(1+\frac{\chi}{\sigma^2}\right)}\right)>0,  \; \forall x \in \R,
	\end{equation*}
	and finally
	\begin{equation}\label{A.61} 
	    |\lambda(x)|\leq \frac{1+\frac{\chi}{\sigma^2}}{c-\chi C_U}  
 \text{ and }|\kappa(x)|\leq \frac{1+\frac{\chi}{\sigma^2}}{c-\chi C_U} .
	\end{equation}
	By using \eqref{A.60}, \eqref{A.61} and \eqref{B.39} in Lemma \ref{LEM3.5} we rewrite   as
	\begin{align*}
		I_1(x)\leq  \frac{1+\frac{\chi}{\sigma^2}}{c-\chi C_U} \int_{[0, x]} \exp\left(\frac{1+\frac{\chi}{\sigma^2}}{c-\chi C_U}|x-s|\right) \int_{[s, x]}C_\lambda(l)\dd l \dd s\Vert U-U_0\Vert_{0, \eta}. 
	\end{align*}
	Thus there exists a continuous function $C_{I_1}(x)$ such that 
	\begin{equation}\label{A.62}
		I_1(x)\leq C_{I_1}(x)\Vert U-U_0\Vert_{0, \eta}.
	\end{equation}
	Next we estimate $I_2(x)$ in \eqref{A.59}. By using \eqref{B.39} and \eqref{A.61}, we have
	\begin{align*}
		I_2(x)&\leq \dis\int_{[0, x]}\left|\kappa_{0}(s)-\kappa(s)\right|\exp\left(\frac{1+\frac{\chi}{\sigma^2}}{c-\chi C_U}|x-s|\right)\ds\\
		&\leq \int_{[0, x]}C_\kappa(s)\exp\left(\frac{1+\frac{\chi}{\sigma^2}}{c-\chi C_U}|x-s|\right)\ds\Vert U-U_0\Vert_{0, \eta},
	\end{align*}
	thus there exists a continuous function $C_{I_2}(x)$ such that 
	\begin{equation}\label{A.63}
		I_2(x)\leq C_{I_2}(x)\Vert U-U_0\Vert_{0, \eta}.
	\end{equation}
	Combining \eqref{3.63}, \eqref{A.62} and \eqref{A.63}, there exists a continuous function $C_I(x):=C_{I_1}(x)+C_{I_2}(x)$ such that \eqref{A.56} holds. Step 2 is completed.
	\medskip

	\noindent\textbf{Conclusion of Part A:} By choosing $\delta_1$ such that 
	\begin{equation} \label{D.64}
		\delta_1:=\dfrac{\varepsilon}{6}\left(\dfrac{1}{\dis\sup_{x\in[-K, K]}\frac{1}{u_0}C_H(x)+C_I(x)}\right),
	\end{equation}
	we conclude from \eqref{A.49}, \eqref{A.52} and \eqref{A.56} that indeed 
	\begin{align*}
		\sup_{x\in[-K, K]}e^{-\eta|x|}|V(x)-V_0(x)|&\leq \sup_{x\in[-K, K]}|V(x)-V_0(x)|\leq \sup_{x\in[-K, K]}\frac{1}{u_0}H(x)+I(x) \\ 
		&\leq \sup_{x\in[-K, K]}\left(\frac{1}{u_0}C_H(x)+C_I(x)\right)\Vert U-U_0\Vert_{0, \eta}\\ 
		&\leq \frac{\varepsilon}{6},
	\end{align*}
	whenever $\Vert U-U_0\Vert_{0, \eta}\leq \delta_1$. Thus \eqref{A.46} holds, and this concludes Part A.\bigskip

	{\noindent\textbf{Part B:} We prove that for each admissible profile $U_0\in\mathcal{A}$ and $\varepsilon>0$, there is $\delta>0$ such that whenever 
	\begin{equation*}
		\Vert U-U_0\Vert_{0, \eta}\leq \delta, 
	\end{equation*}
	we have
	\begin{equation*}
		\Vert V-V_0\Vert_{1, \eta}\leq \varepsilon.
	\end{equation*}
	}
	By Lemma \ref{LEM3.3} we know that $V=\mathcal{T}(U)\in \mathcal{A}$ and $V_0=\mathcal{T}(U_0)\in \mathcal{A}$. Therefore
	\begin{equation*}
		|V'(x)|\leq C_U \text{ and } |V_0'(x)|\leq C_U, \; \forall x \in \R. 
	\end{equation*}
	Let $K>0$ be such that 
	\begin{equation*}
		C_Ue^{-\eta K}\leq \frac{\varepsilon}{12} .
	\end{equation*}
	We have
	\begin{align}
		\nonumber	\sup_{x\in \mathbb R} e^{-\eta |x|}|V'(x)-V_0'(x)| &\leq \sup_{x\leq -K} e^{-\eta |x|}|V'(x)-V_0'(x)| + \sup_{|x|\leq K} e^{-\eta |x|}|V'(x)-V_0'(x)|\\
		\nonumber	&\quad + \sup_{x\geq K}e^{-\eta |x|}|V'(x)-V_0'(x)|\\ 
		\nonumber	&\leq e^{-\eta K}\left(\sup_{x\leq -K}|V'(x)-V_0'(x)|+\sup_{x\geq K}|V'(x)-V_0'(x)|\right) \\
		\nonumber	&\quad + \sup_{|x|\leq K}e^{-\eta |x|}|V'(x)-V_0'(x)| \\ 
		\label{B.65}	& \leq \frac{2\varepsilon}{6}+\sup_{|x|\leq K}e^{-\eta |x|}|V'(x)-V_0'(x)|.
	\end{align}
	Thus  there remains only to  establish that
	\begin{equation}\label{A.65}
		\sup_{|x|\leq K}e^{-\eta |x|}|V'(x)-V_0'(x)|\leq \frac{\varepsilon}{6}.
	\end{equation}
	We note that $V$ and $V_0$ satisfy \eqref{3.2}, therefore
	\begin{equation}\label{3.74}
		V'(x)=\lambda(x)V(x)- \kappa(x)V^2(x),\text{ and } V_0'(x) = \lambda_0(x)-\kappa_0(x), \text{ for all }x\in\R.
	\end{equation}
	Then we have that
	\begin{equation*}
		\begin{aligned}
			|V'(x)-V_0'(x)|&=|\lambda(x)V(x)- \kappa(x)V^2(x)-\lambda_0(x)V_0(x)+\kappa_0(x)V_{0}^2(x)|\\
			&=|(\lambda(x)-\lambda_{0}(x))V(x)+\lambda_{0}(x)(V(x)-V_{0}(x))\\
			&\quad+(\kappa_{0}(x)-\kappa(x))V_{0}^2(x)+\kappa(x)(V_{0}^2(x)-V^2(x))|\\
			&\leq|\lambda(x)-\lambda_{0}(x)||V(x)|+\lambda_{0}(x)|V(x)-V_{0}(x)|\\
			&\quad+|\kappa_{0}(x)-\kappa(x)|V_{0}^2(x)+\kappa(x)|V_{0}^2(x)-V^2(x)|\\
			&=|\lambda(x)-\lambda_{0}(x)||V(x)|+\lambda_{0}(x)|V(x)-V_{0}(x)|\\
			&\quad+|\kappa_{0}(x)-\kappa(x)|V_{0}^2(x)+\kappa(x)|V_{0}(x)-V(x)||V_{0}(x)+V(x)|.
		\end{aligned}
	\end{equation*}
	By using the fact that $0\leq V(x)\leq1$ and $0\leq V_0(x)\leq1$ for $x\in\R$, we have that
	\begin{equation}\label{3.75}
		\begin{aligned}
			|V'(x)-V_0'(x)|&\leq|\lambda(x)-\lambda_{0}(x)|+\lambda_{0}(x)|V(x)-V_{0}(x)|+|\kappa_{0}(x)-\kappa(x)|+2\kappa(x)|V_{0}(x)-V(x)|.
		\end{aligned}
	\end{equation}
	It follows from \eqref{A.43}, \eqref{A.61}, \eqref{A.56}, \eqref{A.52} and \eqref{B.39} and \eqref{B.40}  that
	\begin{equation}\label{3.76}
	    |V'(x)-V_0'(x)|\leq\left(C_\lambda(x)+\frac{1+\frac{\chi}{\sigma^2}}{u_0(c-\chi C_U)}(C_H(x)+u_0C_I(x))+C_\kappa(x)+2\frac{1+\frac{\chi}{\sigma^2}}{c-\chi C_U} \right)\Vert U-U_0\Vert_{0, \eta}.
	\end{equation}
	Let 
	\begin{equation}\label{A.71}
		\delta_2:=\frac{\varepsilon}{3}\left[\sup_{x\in [-K, K]}\left(C_\lambda(x)+\frac{1+\frac{\chi}{\sigma^2}}{u_0(c-\chi C_U)}(C_H(x)+u_0C_I(x))+C_\kappa(x)+2\frac{1+\frac{\chi}{\sigma^2}}{c-\chi C_U} \right)e^{-\eta |x|}\right]^{-1}, 
	\end{equation}
	then whenever $\Vert U-U_0\Vert_{0, \eta}\leq \delta_2$ we have
	\begin{equation*}
		\sup_{x\in[-K, K]}e^{-\eta |x|} |V'(x)-V_0'(x)| \leq \frac{\varepsilon}{3},  
	\end{equation*}
	therefore, recalling \eqref{B.65},
	\begin{equation}\label{A.70}
		\sup_{x\in\mathbb R} e^{-\eta |x|}|V'(x)-V_0'(x)| \leq \frac{\varepsilon}{2}.
	\end{equation}\medskip

	\noindent\textbf{Conclusion of Part B:}	
	Let $\delta:=\min(\delta_1, \delta_2)$ where $\delta_1$ is defined in \eqref{D.64} and $\delta_2$ is defined in \eqref{A.71}. Then if $\Vert U-U_0\Vert_{0, \eta}\leq \delta $ we know from Part A  \eqref{B.46} that 
	\begin{equation*}
		\sup_{x\in\mathbb R}e^{-\eta |x|}|V(x)-V_0(x)|\leq \frac{\varepsilon}{2}, 
	\end{equation*}
	and from \eqref{A.70} that
	\begin{equation*}
		\sup_{x\in\mathbb R} e^{-\eta |x|}|V'(x)-V_0'(x)| \leq \frac{\varepsilon}{2}.
	\end{equation*}
	So finally 
	\begin{equation*}
		\Vert V-V_0\Vert_{1, \eta}=\sup_{x\in\mathbb R}e^{-\eta |x|}|V(x)-V_0(x)| + \sup_{x\in\mathbb R} e^{-\eta |x|}|V'(x)-V_0'(x)| \leq \varepsilon.
	\end{equation*}
	Part B is proved. Since we always have $\Vert U-U_0\Vert_{0, \eta}\leq \Vert U-U_0\Vert_{1, \eta}$, the continuity holds for the norm $\Vert\cdot\Vert_{1, \eta}$. 
	Lemma \ref{LEM3.5} is proved. 
\end{proof}

\section{Proof of Theorem \ref{TH1.3}}
\label{Section4}
	From the definition of admissible functions $\mathcal{A}$, it is a nonempty, closed, convex, bounded subset of the Banach space $BUC^1_\eta \left( \R \right)$. By Lemmas \ref{LEM3.6} and \ref{LEM3.7}, we obtain that $\mathcal{T}$ is a continuous compact operator on $\mathcal{A}$. Therefore, by the Schauder fixed-point theorem, there exists $U$ in $\mathcal{A}$ such that
	$$\mathcal{T}(U) = U.$$
	Applying Lemma \ref{LEM3.3}, we have that $U\in C^1(\R)$ and {$0\leq U'(x)\leq C_U$} for any $x\in\R$.
	Therefore, we have that
	\begin{equation*}
		U(x)=\frac{u_0\ee^{\int_{0}^{x}\lambda(s)\ds}}{1+u_0\int_{0}^{x}\kappa(s)\ee^{\int_{0}^{s}\lambda(l)\dl}\ds},
	\end{equation*}
	wherein
	$$\lambda(x)=\frac{1+\frac{\chi}{\sigma^{2}}P(x) }{c-\chi P'(x)},$$
	and
	$$\quad\kappa(x)=\frac{1+ \frac{\chi}{\sigma^{2}}}{c-\chi P'(x)},$$
	and $P(x)$ is the unique solution of the elliptic equation
	\begin{equation}\label{4.2}
		P(x)-\sigma^{2} P''(x)=U(x), \forall x \in \mathbb{R}.
	\end{equation}
	Namely, we have that
	\begin{align}\label{4.3}
		U'(x)&=\frac{1}{c-\chi P'(x)}U(x) \left( \left( 1+\frac{\chi}{\sigma^{2}} P(x)\right)-\left(1+ \frac{\chi}{\sigma^{2}} \right)U(x) \right), \forall x \in \mathbb{R}.
	\end{align}
	Therefore, we have that
	\begin{align*}
		cU'(x)-\chi P'(x)U'(x)-\frac{\chi}{\sigma^{2}}U(x)( P(x)-U(x))=U(x)(1-U(x)), \forall x \in \mathbb{R}.
	\end{align*}
	By using \eqref{4.2}, we have that
	\begin{equation}\label{4.4}
		cU'(x)-\chi (P'(x)U(x))'=U(x)(1-U(x)), \forall x \in \mathbb{R}.
	\end{equation}
	We prove that
	$$U(\infty):=\lim_{x\rightarrow+\infty}U(x)=1 \text{ and }U(-\infty):=\lim_{x\rightarrow-\infty}U(x)=0.$$
	Indeed, since $U'(x)\geq0$ and $0<U(x)\leq1$ for any $x\in\R$, then $U(\infty)$ exists. By using $P$-equation \eqref{1.7}, the function $x \to P(x)$ is increasing and bounded, and by Lebesgue's dominated convergence theorem, we have  
	$$
	P(\pm \infty)=U( \pm \infty).
	$$ 
	Therefore,
	\begin{equation}\label{4.5}
		\lim_{x\rightarrow \pm \infty}U'(x)=0, \text{ and }\lim_{x\rightarrow \pm \infty}P'(x)=0.
	\end{equation}
	It follows from \eqref{4.4} and \eqref{4.5}, we have that
	\begin{equation*}
		\lim_{x\rightarrow \pm \infty}U(x)(1-U(x))=0,
	\end{equation*}
	and since $x \to U(x)$ increasing and $U(0)=u_0>0$, this implies that 
	$$
	U( - \infty)=0, \text{ and }U( + \infty)=1.
	$$ 
	This completes the proof of the Theorem \ref{TH1.3}.

\section{Numerical simulations}
\label{Section5}

We choose a bounded interval $[-K,K]$ and an initial distribution $u_0\in C([-K,K])$ as follows
\begin{equation}\label{5.1}
	\dis u_0(x)=\dfrac{2\, \ee^{\dis-\beta(x+K)}}{1+\ee^{\dis-\beta(x+K)}}. 
\end{equation}
In the following numerical simulations, we solve the PDE numerically using the upwind scheme, and we refer to Leveque \cite{Leveque} and Toro \cite{Toro} for more results on this subject. The numerical method used for the simulations is presented in Section \ref{SectionA} of the Appendix. 

In this section, we set the parameters of the system \eqref{A.1} all equal to one. That is 
$$
\sigma=\chi=\lambda=\kappa=1. 
$$
In Figure \ref{Fig3} we plot $x \to u_0(x)$ with the parameter values $\beta=1$, and $K=20$,  and the corresponding traveling wave profile which coincides with $x \to u(20,x)$ the solution of system \eqref{A.1}  at $t=20$ days.  

\begin{figure}[H]
	\centering
	\includegraphics[scale=0.17]{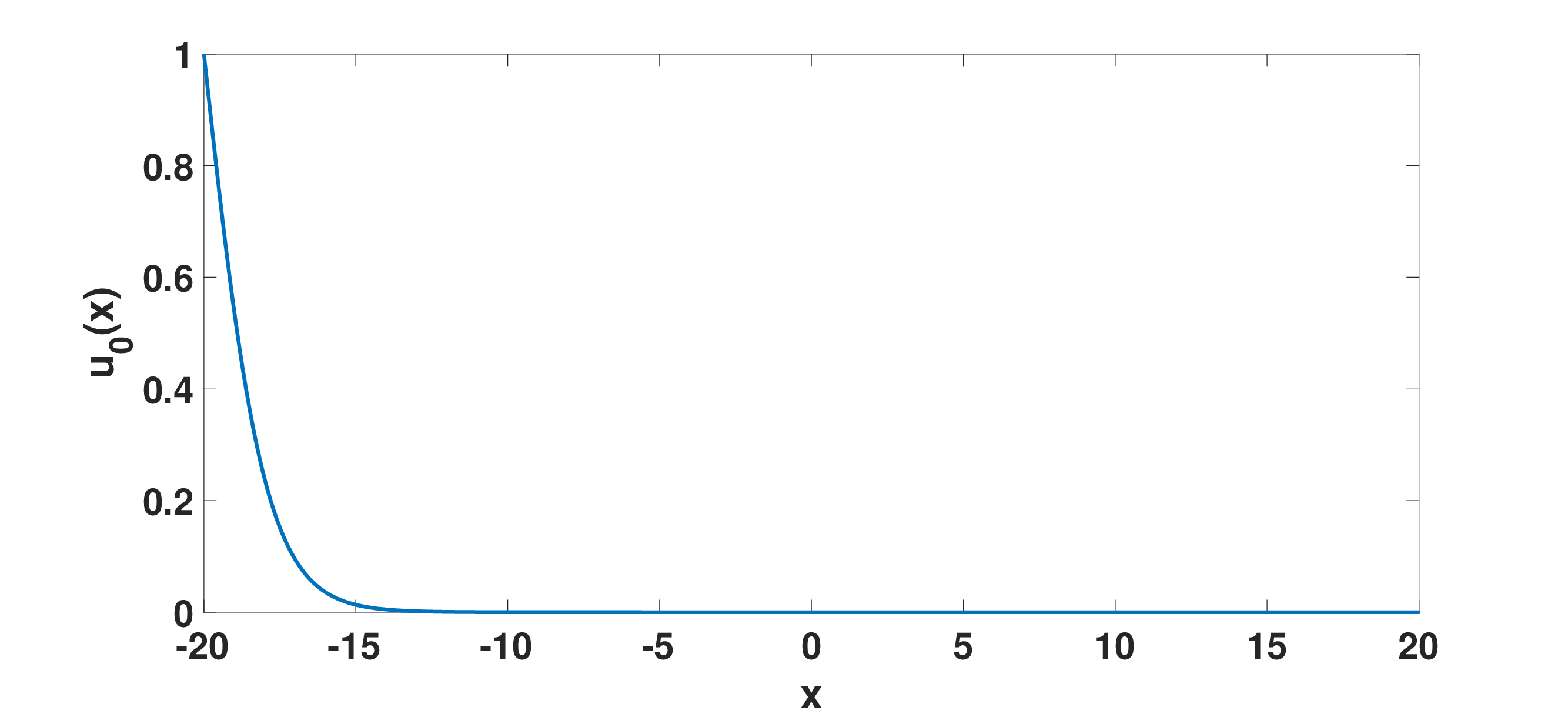}
\includegraphics[scale=0.17]{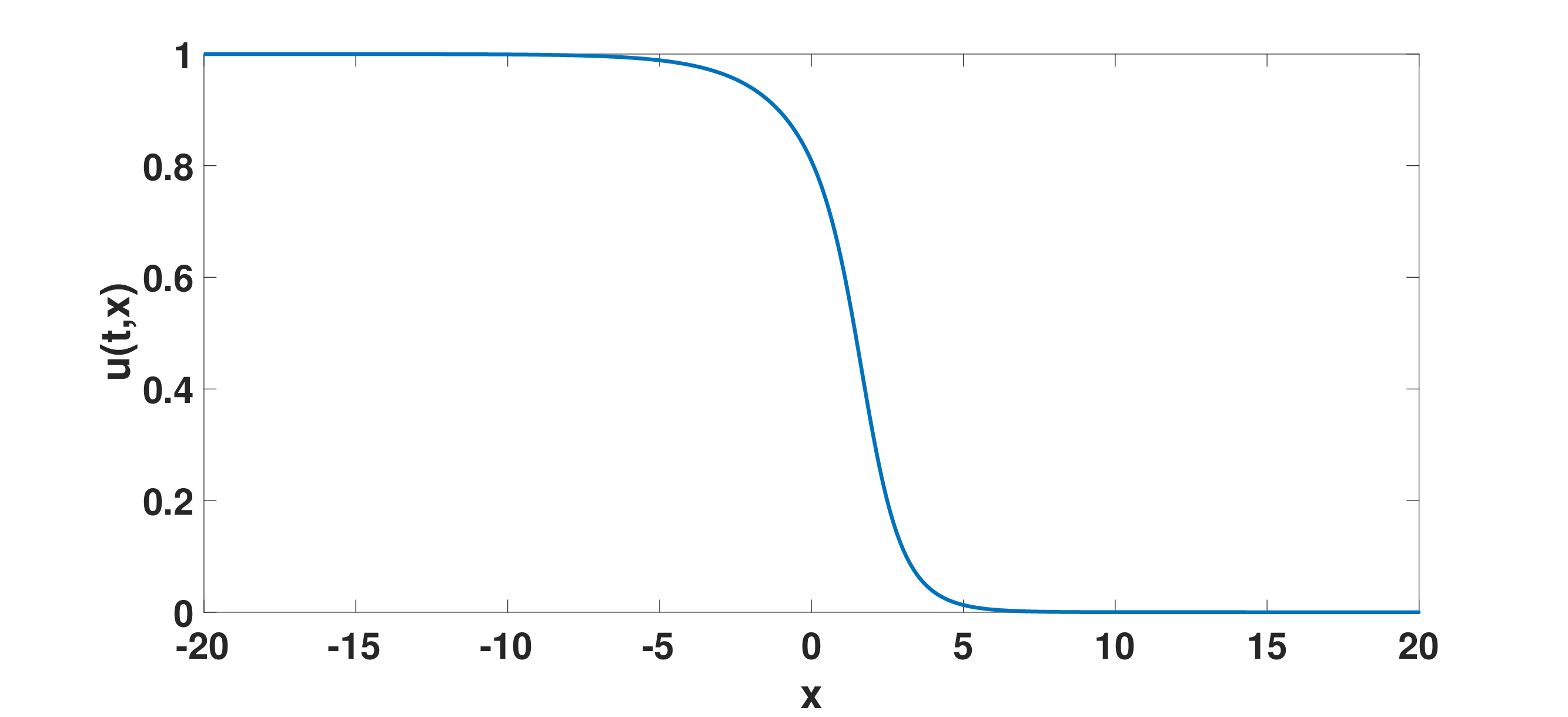}
	\caption{{\textit{On the left-hand side, we plot $x \to u_0(x)$ the initial distribution  of system \eqref{A.1}, obtained by using  formula \eqref{5.1} with $\beta=1$ and $K=20$. On the right-hand side, we plot the traveling wave profile which coincides with $x \to u(t,x)$ the solution of system \eqref{A.1}  at $t=20$ days.}}}\label{Fig3}
\end{figure}

In Figure \ref{Fig4}, we run a simulation from $t=0$ until $t=20$ of the model \eqref{A.1}. We observe that the traveling wave appears almost immediately after the starting time $t=0$.

\begin{figure}[H]
	\centering
	\includegraphics[scale=0.17]{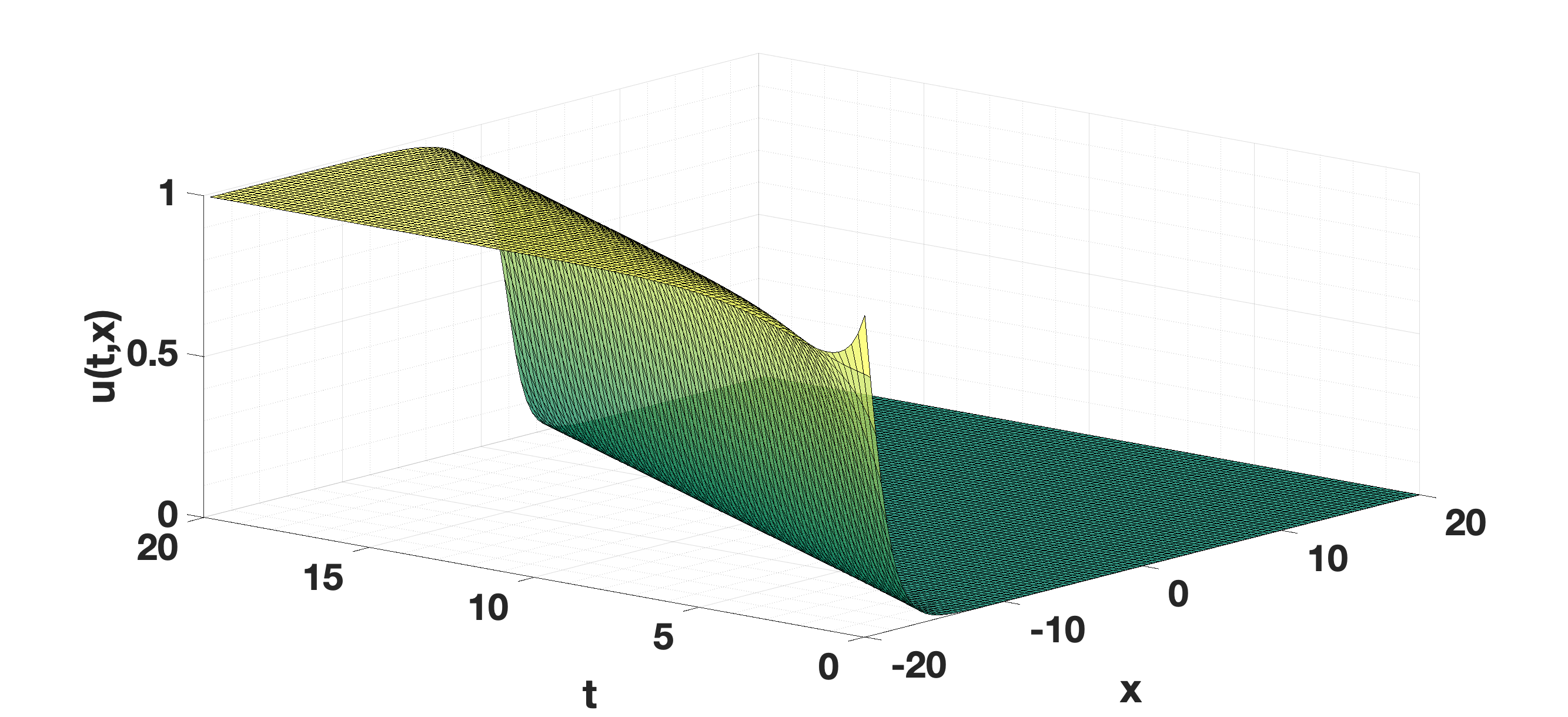}
	\includegraphics[scale=0.17]{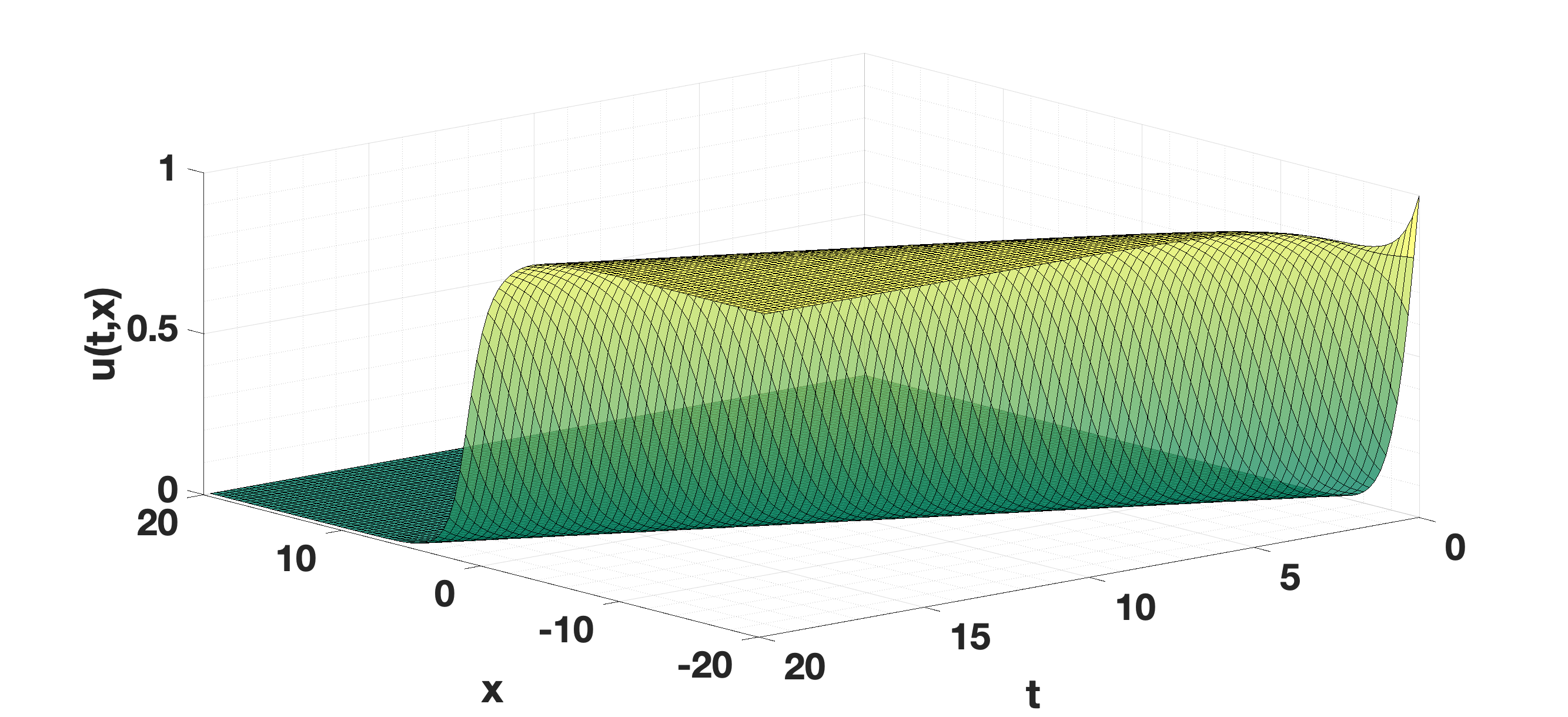}
	\caption{{\textit{In this figure, we plot the solution of the model \eqref{A.1} starting from the initial distribution \eqref{5.1} (with  $\beta=1$ and $K=20$).}}}\label{Fig4}
\end{figure}

Next we use the following initial value 
\begin{equation}\label{5.2}
	\dis u_0(x)=\max \left( 1- \beta \left(x+K\right), 0\right). 
\end{equation}
In Figure \ref{Fig5}, we plot $x \to u_0(x)$ the initial distribution  of system \eqref{A.1} (on the left-hand side), and the corresponding traveling wave profile which coincides with $x \to u(20,x)$ the solution of system \eqref{A.1}  at $t=20$ days. 

\begin{figure}[H]
	\centering
	\includegraphics[scale=0.17]{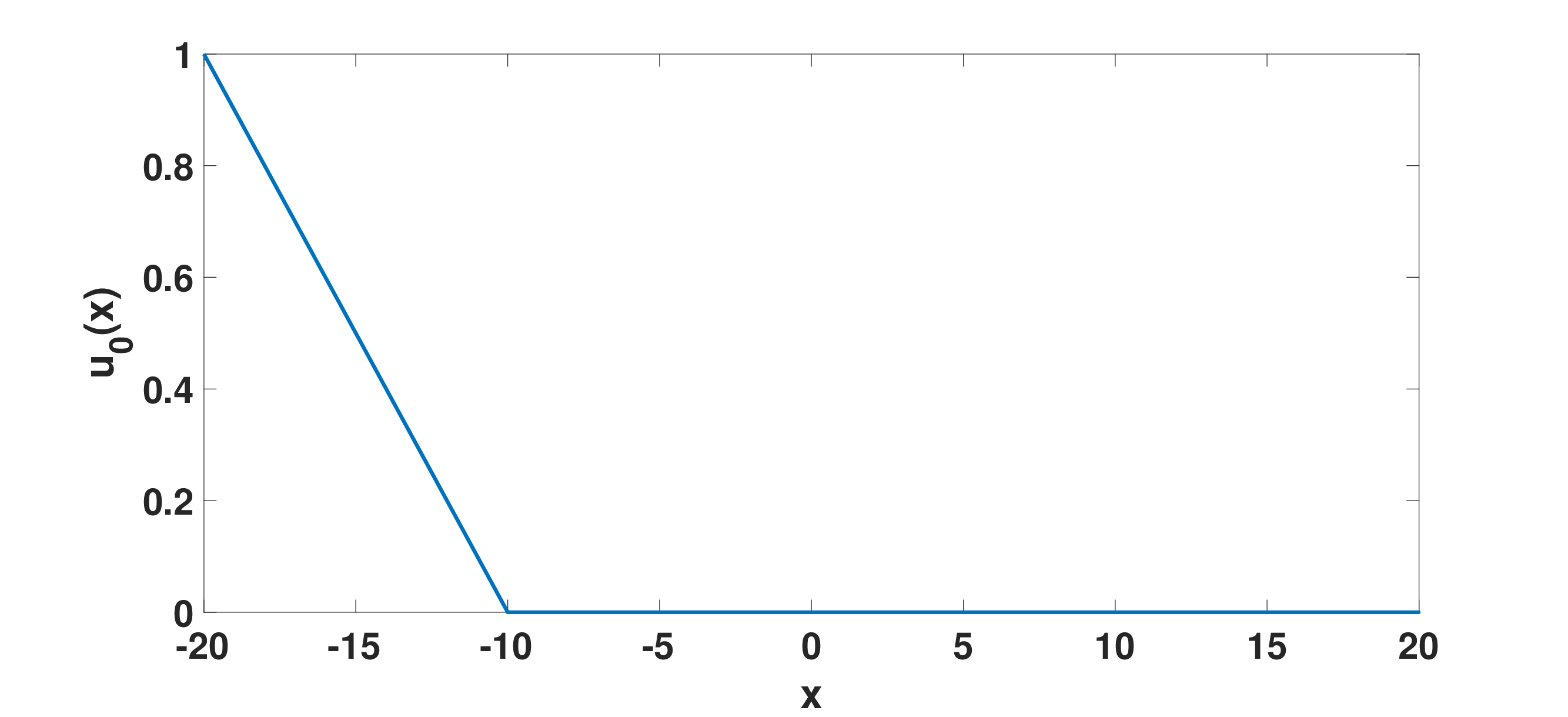}
	\includegraphics[scale=0.17]{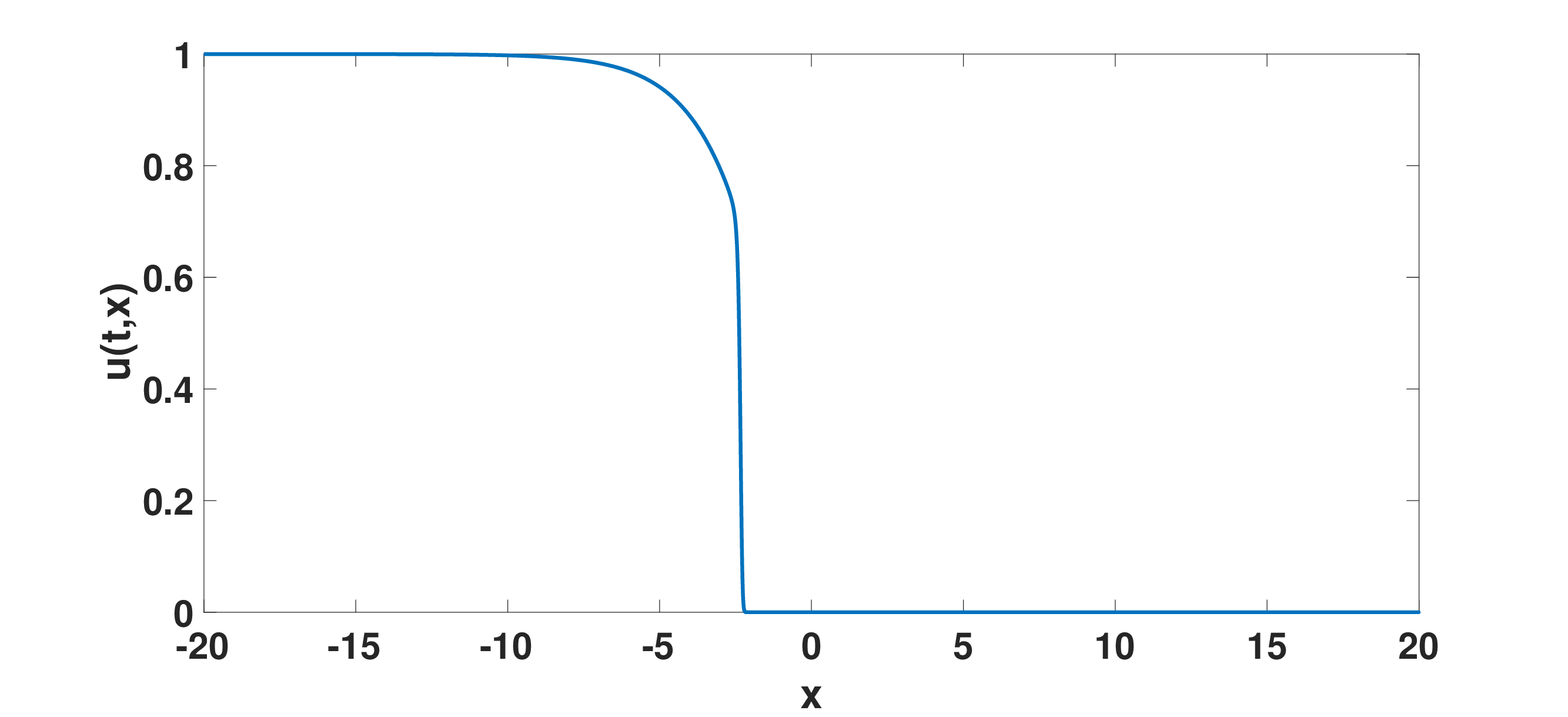}
	\caption{{\textit{On the left-hand side, we plot $x \to u_0(x)$ the initial distribution  of system \eqref{A.1}, obtained by using  formula \eqref{5.2} with $\beta=0.1$ and $K=20$. On the right-hand side, we plot the traveling wave profile which coincide with $x \to u(t,x)$ the solution of system \eqref{A.1}  at $t=20$ days.}}}\label{Fig5}
\end{figure}

In Figure \ref{Fig6}, we run a simulation from $t=0$ until $t=20$ of the model \eqref{A.1}. We observe that the traveling wave appears almost immediately after the starting time $t=0$.

\begin{figure}[H]
	\centering
	\includegraphics[scale=0.17]{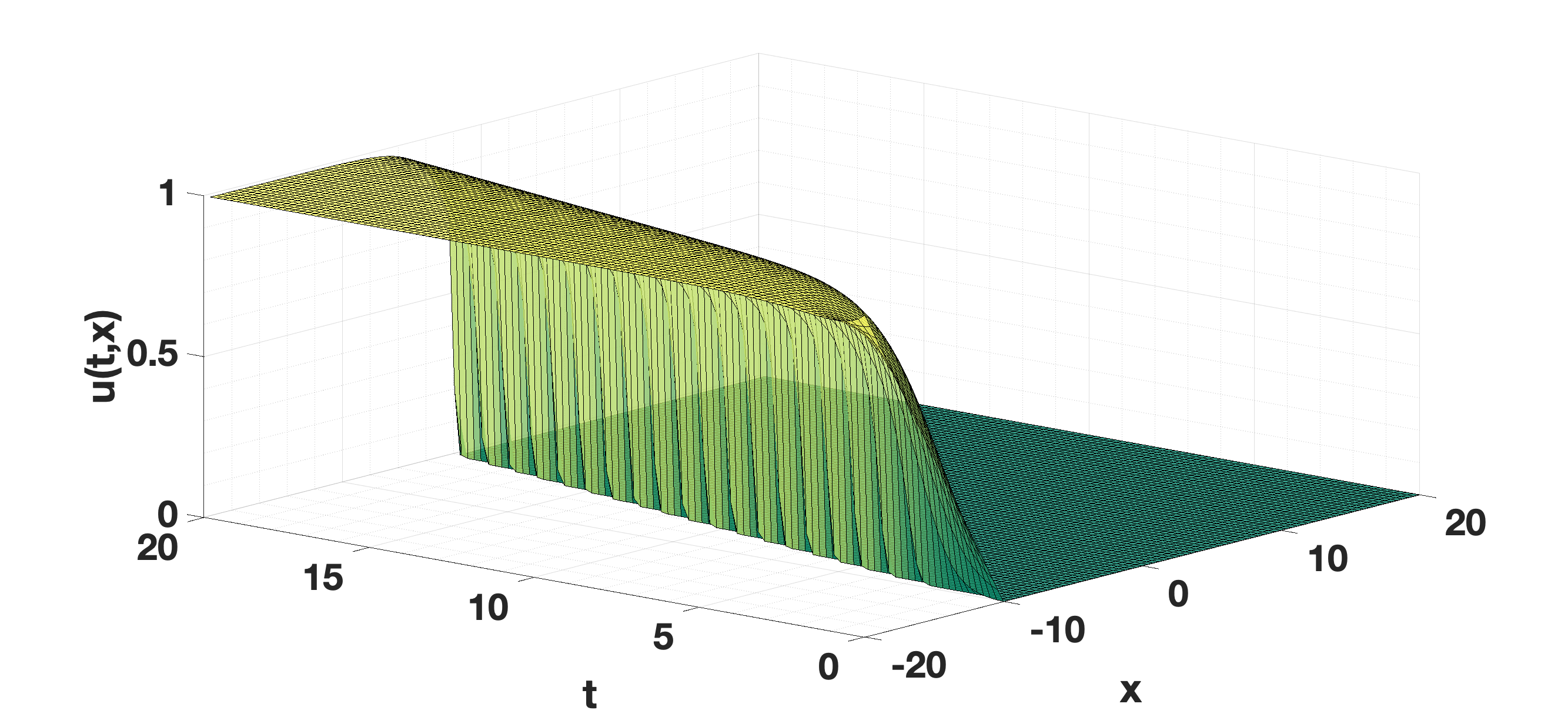}
	\includegraphics[scale=0.17]{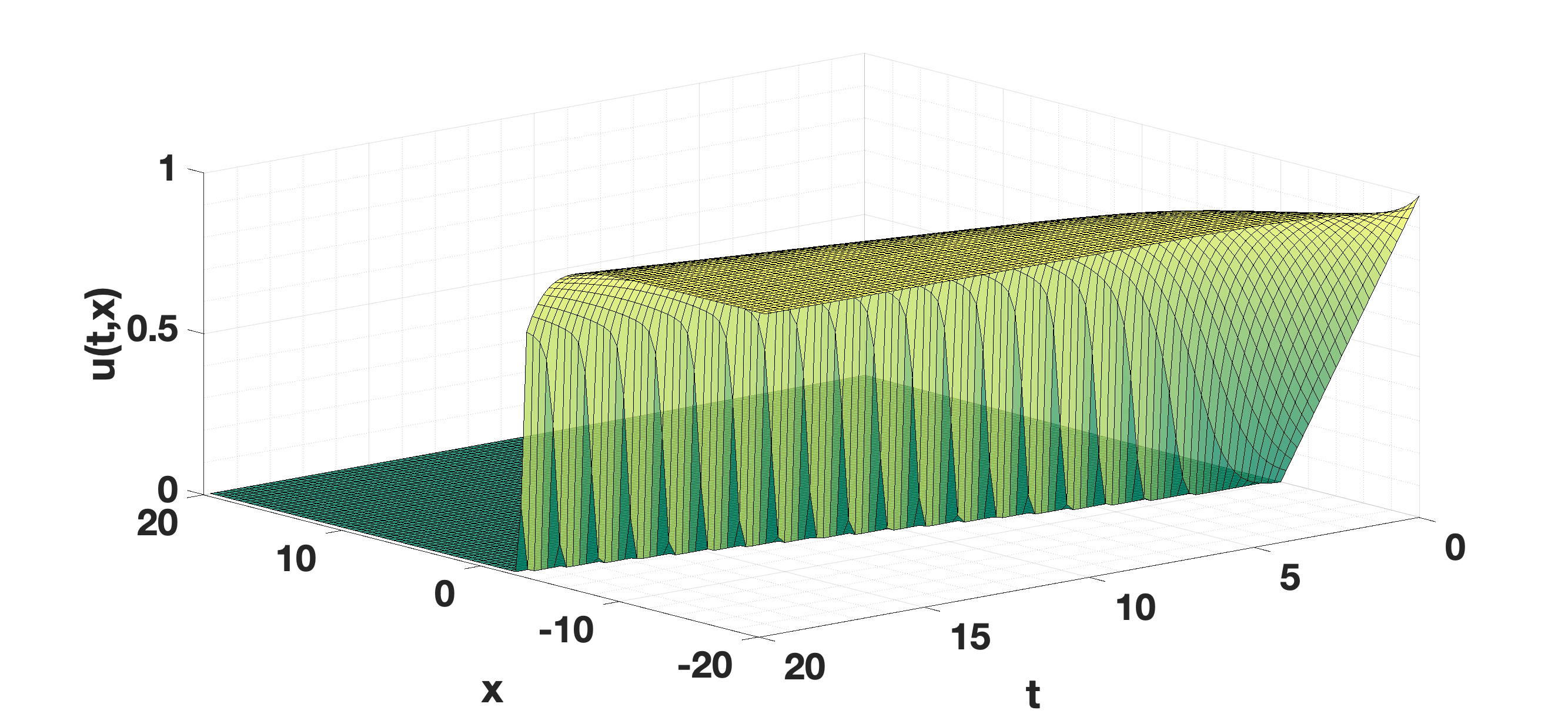}
	\caption{{\textit{In this figure, we plot the solution of the model \eqref{A.1} starting from the initial distribution \eqref{5.2} (with   $\beta=0.1$ and $K=20$).}}}\label{Fig6}
\end{figure}

On the one hand, our numerical simulations show that continuous traveling waves can be observed from an initial distribution decaying exponentially (slowly enough). On the other hand, sharp traveling waves can also be observed when starting the PDE with initial distributions equal to zero on the half-plane. So in practice, both types of traveling waves can be observed numerically. 

Now concerning the traveling speed, we observe numerically that sharp traveling waves are slower than continuous traveling waves. {In this aspect the situation is somehow similar to what is observed with reaction-diffusion equations (like the Fisher-KPP equation), in that the ``slowest'' wave is caught by starting from compactly supported initial data.} The question of the minimal speed is quite intricate given the nonlinear nature of the equation, and we leave it for future works. 

\section{Application to wound healing}
\label{Section6}

The wound healing assay is used in a range of disciplines to study the coordinated movement of a cell population. We refer to the paper of Jonkman et al. \cite{Jonkman} for a review on this topic. In this paper, we consider the cell-cell repulsion described by nonlinear diffusion, but cell-cell attraction also occurs and this problem was recently considered  by Webb \cite{Webb} (see also the references therein for more results). 
\begin{figure}[H]
	\centering
	\includegraphics[scale=0.9]{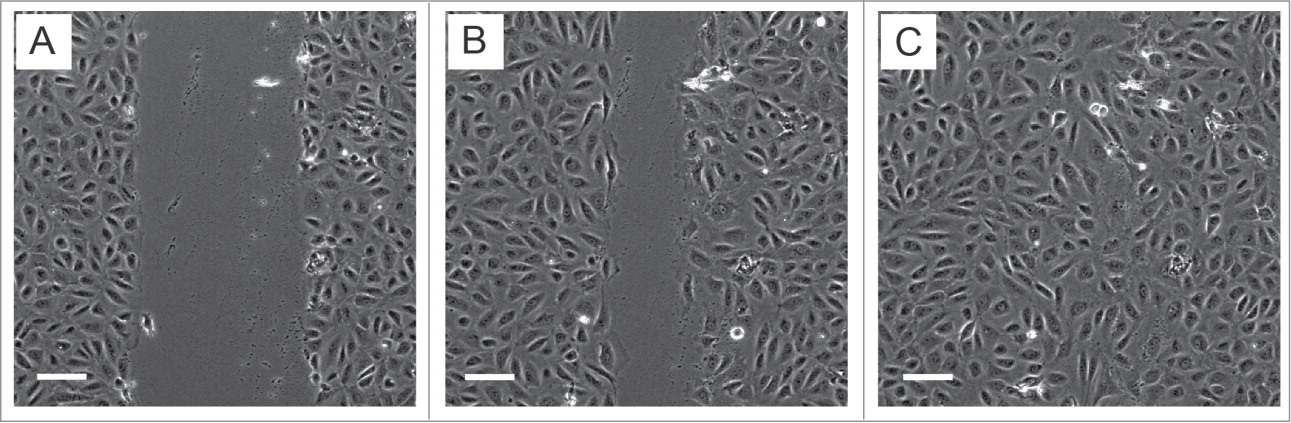}
	\caption{{\textit{Images from a scratch assay experiment at different time points. Human umbilical vein endothelial cells (HUVEC) were plated on gelatin-coated plastic dishes, wounded with a p20 pipette tip, and then imaged overnight using a microscope equipped with point visiting and live-cell apparatus. Scale bar = 120 $\mu$m. This figure is taken from Jonkman et al. \cite{Jonkman}.}}}\label{Fig7}
\end{figure}
In this section, we set the parameters of the system \eqref{A.1} as follows 
$$
\chi=\lambda=4 \text{ and }\sigma=\kappa=1. 
$$
\noindent  \textbf{Initial distribution for imperfect wound:} We choose a bounded interval $[-K,K]$ and an initial distribution $u_0\in C([-K,K])$ as follows
\begin{equation}\label{6.1}
	\dis u_0(x)=\dfrac{1}{2} \left( \dfrac{2\, \ee^{\dis-\beta(x+K)}}{1+\ee^{\dis-\beta(x+K)}} \right)+ \dfrac{1}{2}  \left(  \dfrac{2\, \ee^{\dis-\beta(K-x)}}{1+\ee^{\dis-\beta(K-x)}}  \right). 
\end{equation}
In Figure \ref{Fig9} we plot $x \to u_0(x)$ with the parameter values $\beta=0.5$, and $K=20$,  and  $x \to u(7,x)$ the solution of system \eqref{A.1}  at $t=7$ days. 

\begin{figure}[H]
	\centering
	\includegraphics[scale=0.17]{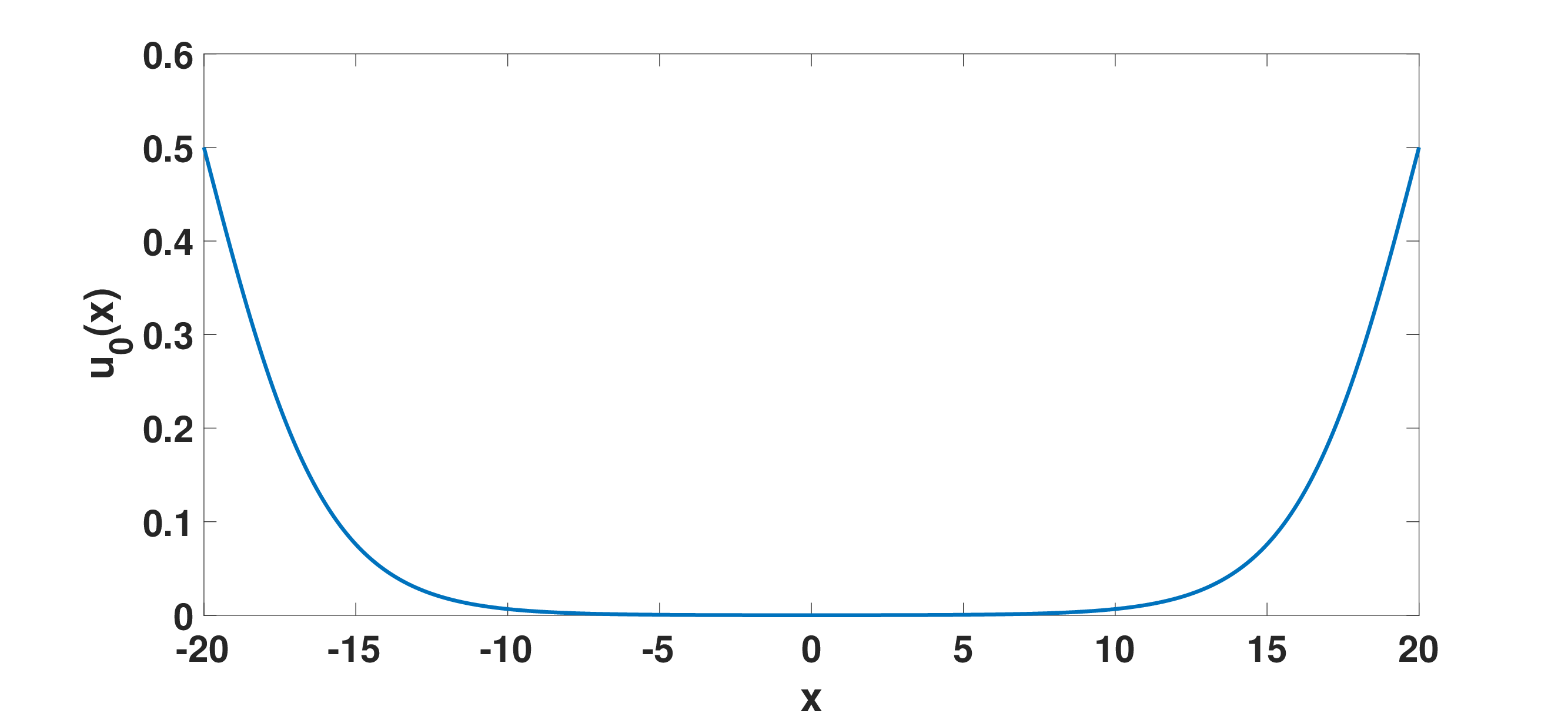}
	\includegraphics[scale=0.17]{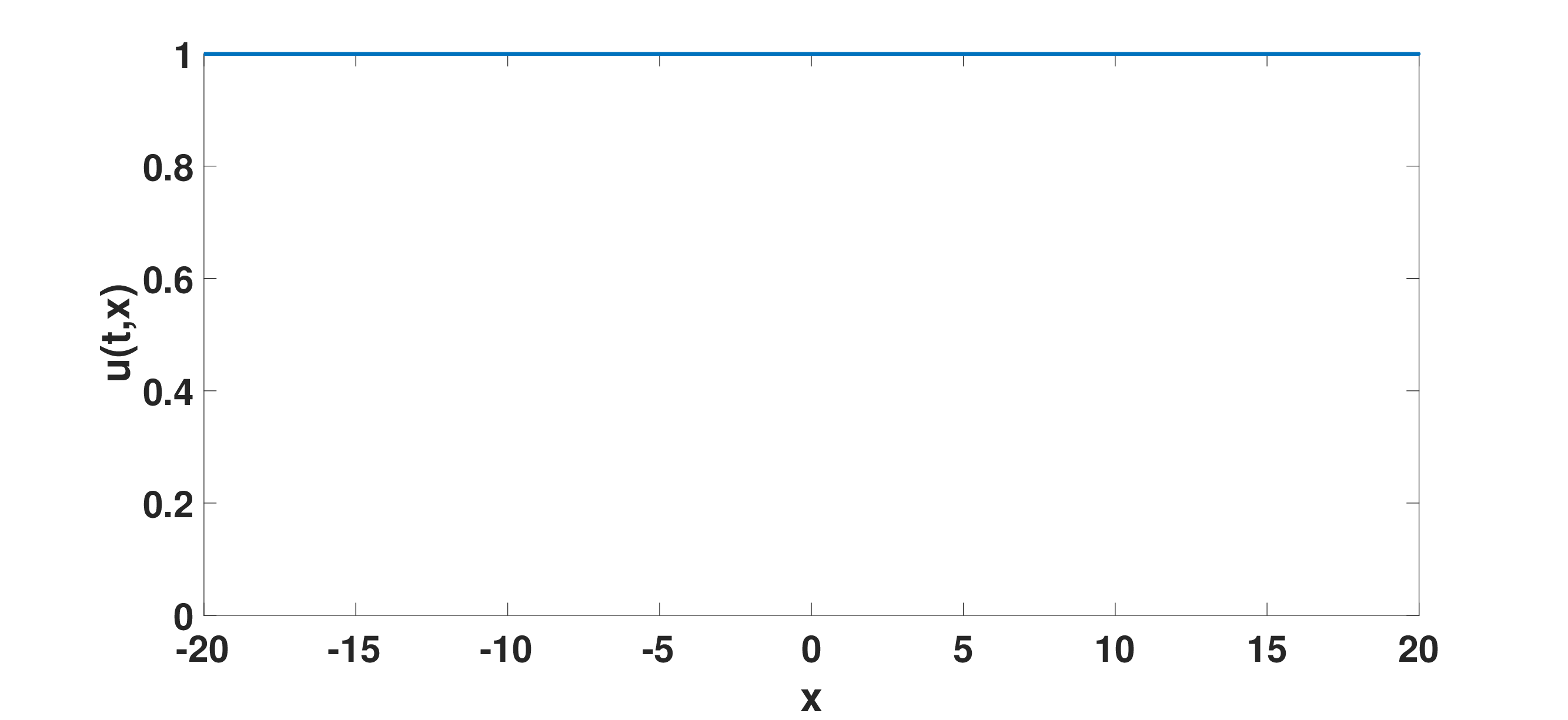}
	\caption{{\textit{On the left-hand side, we plot $x \to u_0(x)$ the initial distribution  of system \eqref{A.1}, obtained by using  formula \eqref{6.1} with $\beta=0.5$ and $K=20$. On the right-hand side, we plot $x \to u(t,x)$ the solution of system \eqref{A.1}  at $t=7$ days. }}}\label{Fig9}
\end{figure}

	In Figure \ref{Fig10}, we run a simulation from $t=0$ until $t=7$ of the model \eqref{A.1}. We observe that two traveling wave moving in opposite directions appears almost immediately after the starting time $t=0$. They merge together to give a flat distribution  approximately on day $2$.

\begin{figure}[H]
	\centering
	\includegraphics[scale=0.17]{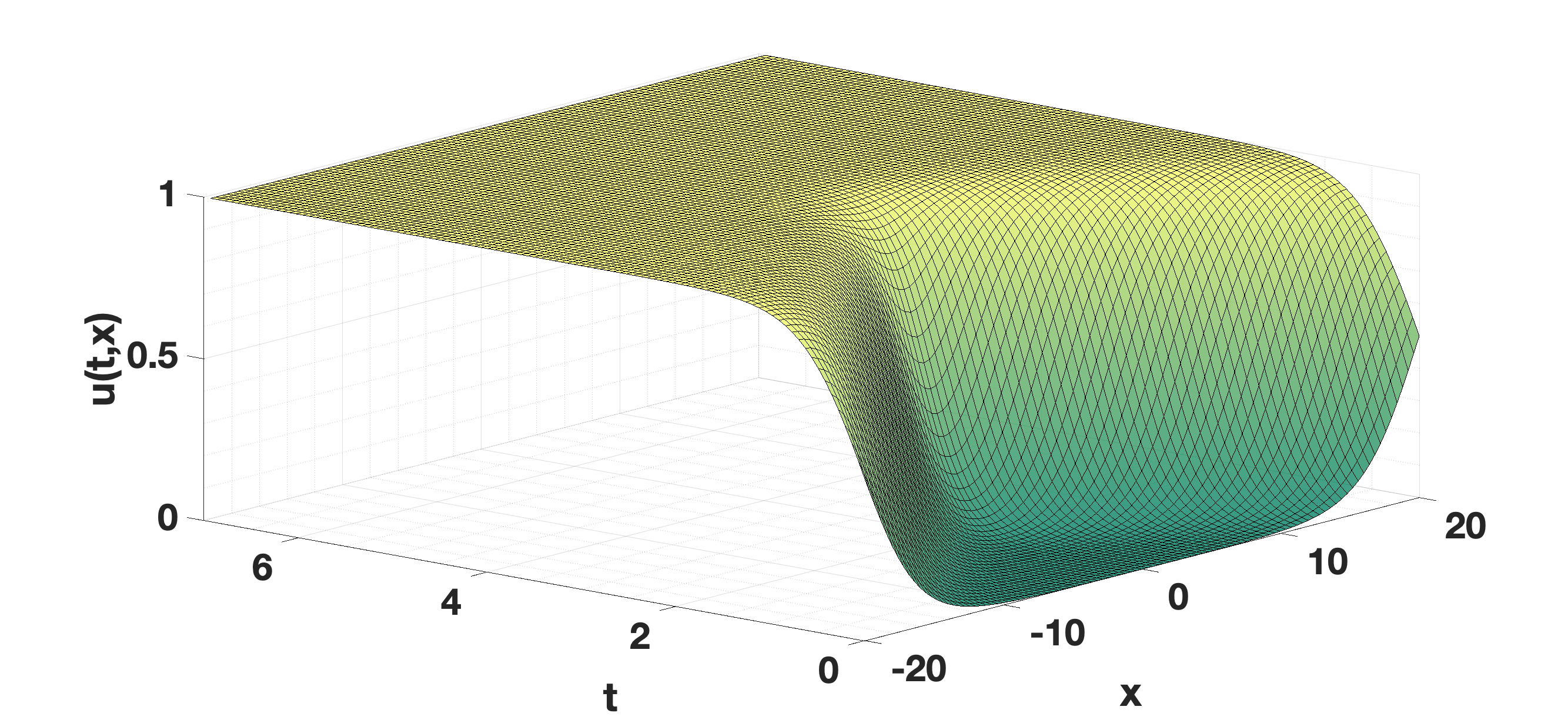}
	\includegraphics[scale=0.17]{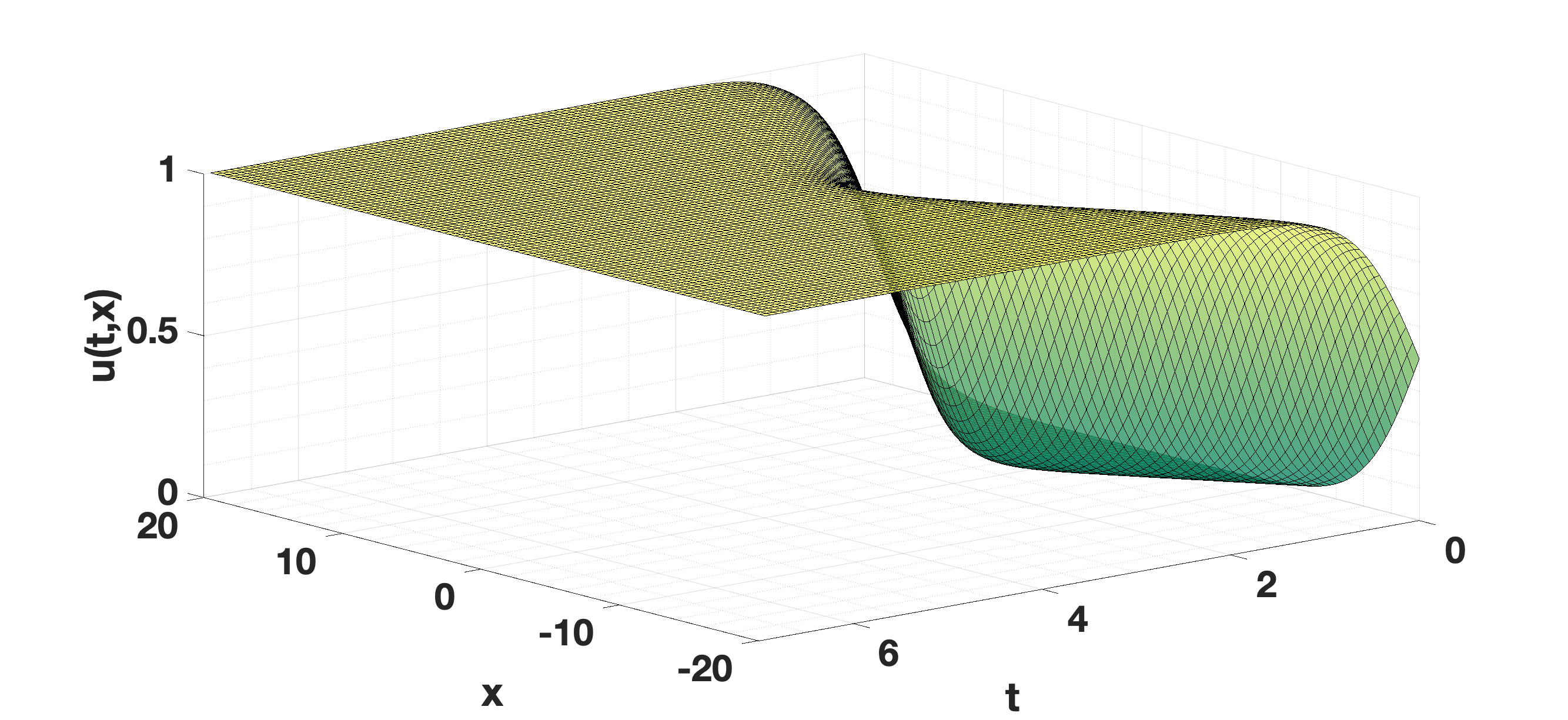}
	\caption{{\textit{In this figure, we plot the solution of the model \eqref{A.1} starting from the initial distribution \eqref{6.1} (with  $\beta=0.5$ and $K=20$).}}}\label{Fig10}
\end{figure}

\noindent \textbf{Initial distribution for  perfect wound:} We choose a bounded interval $[-K,K]$ and an initial distribution $u_0\in C([-K,K])$ as follows
\begin{equation}\label{6.2}
	\dis u_0(x)=\dfrac{1}{2} \left(\max \left( 1- \beta \left(x+K\right), 0\right) \right)+ \dfrac{1}{2}  \left( \max \left( 1- \beta \left(K-x\right), 0\right)  \right). 
\end{equation}
In Figure \ref{Fig11} we plot $x \to u_0(x)$ with the parameter values $\beta=0.07$, and $K=20$,  and  $x \to u(7,x)$ the solution of system \eqref{A.1}  at $t=7$ days. 

\begin{figure}[H]
	\centering
	\includegraphics[scale=0.17]{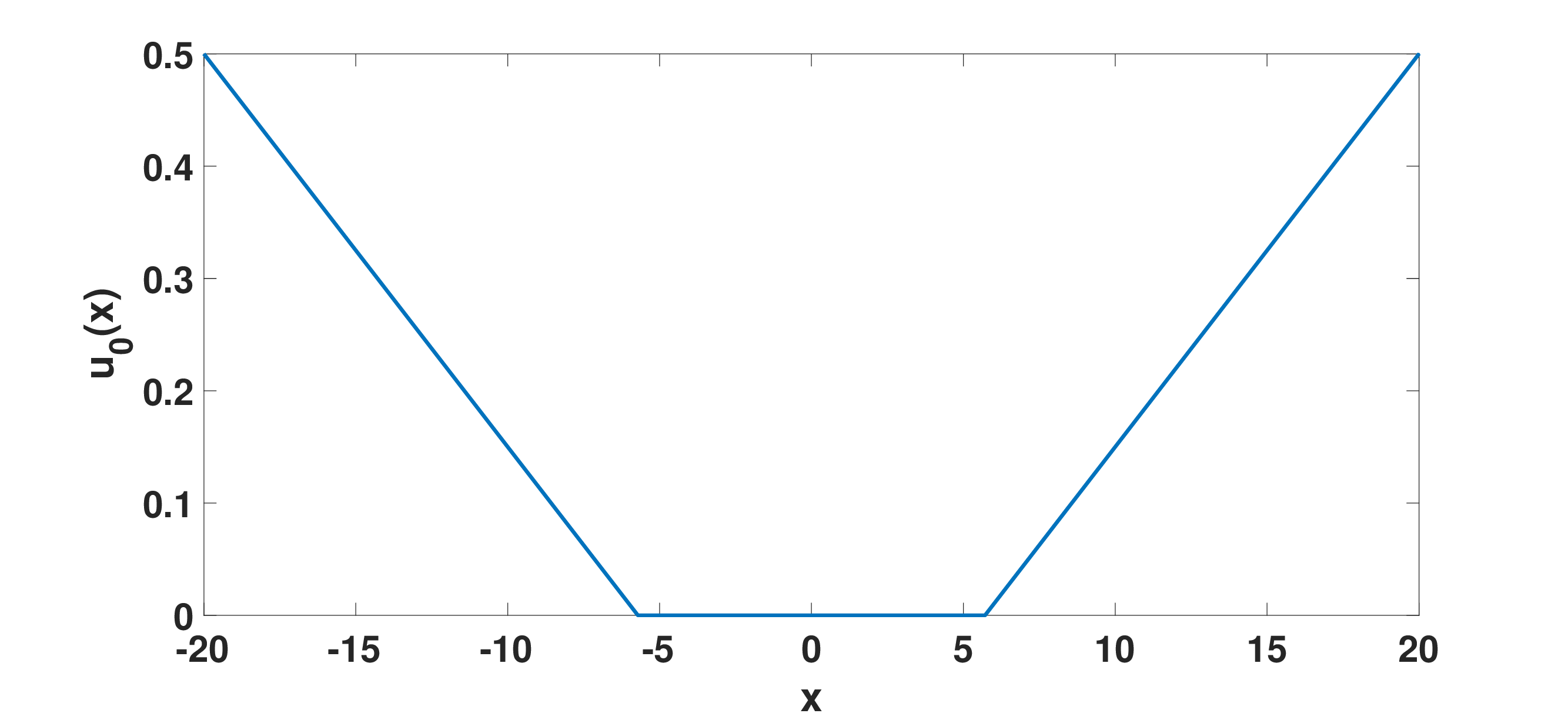}
	\includegraphics[scale=0.17]{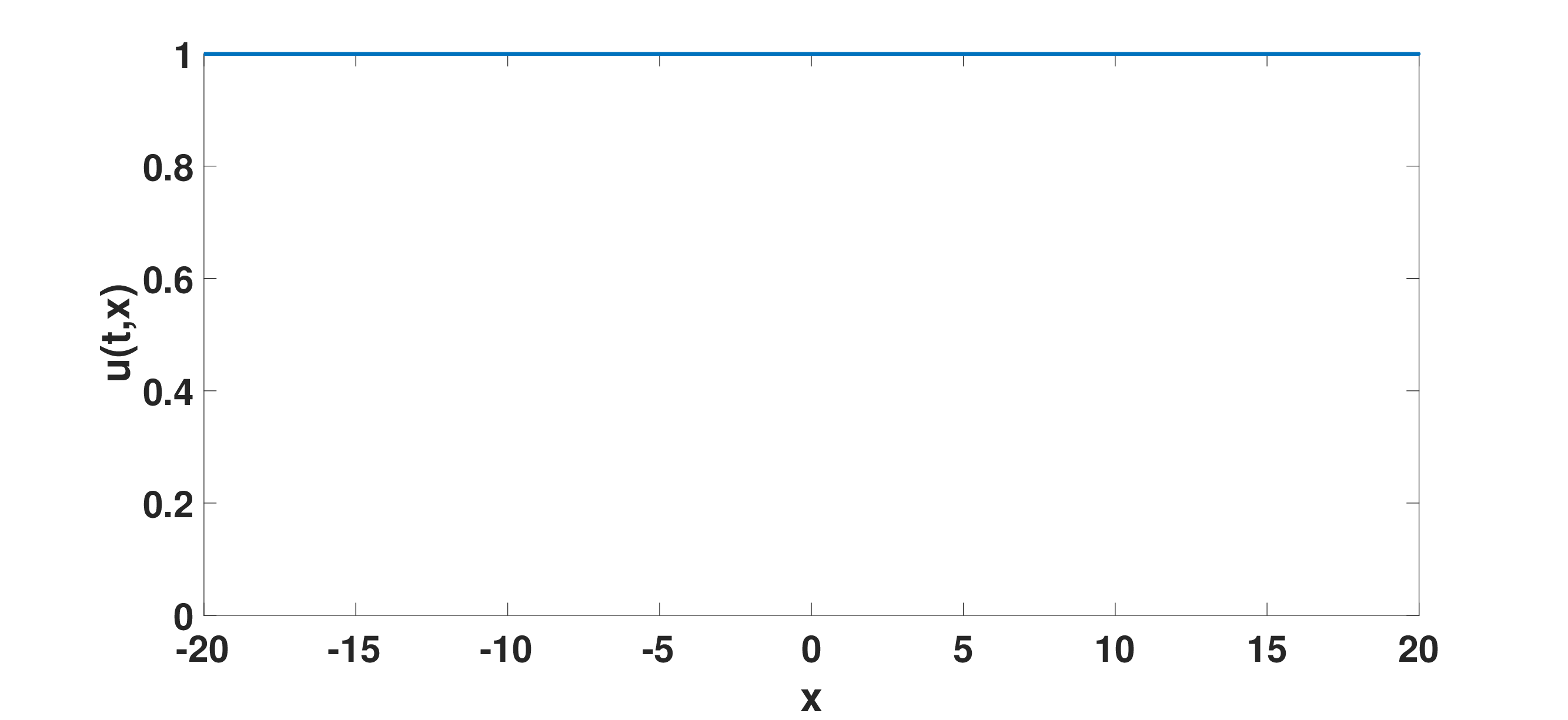}
	\caption{{\textit{On the left-hand side, we plot $x \to u_0(x)$ the initial distribution  of system \eqref{A.1}, obtained by using  formula \eqref{6.2} with $\beta=0.07$ and $K=20$. On the right-hand side, we plot  $x \to u(t,x)$ the solution of system \eqref{A.1}  at $t=7$ days.}}}\label{Fig11}
\end{figure}

In Figure \ref{Fig12}, we run a simulation from $t=0$ until $t=7$ of the model \eqref{A.1}  for the parameter values $\sigma=1$, and $\chi=1$. We observe that two traveling wave  moving in opposite directions appears almost immediately after the starting time $t=0$. They merge together to give a flat distribution  approximately on day $5$.

\begin{figure}[H]
	\centering
	\includegraphics[scale=0.17]{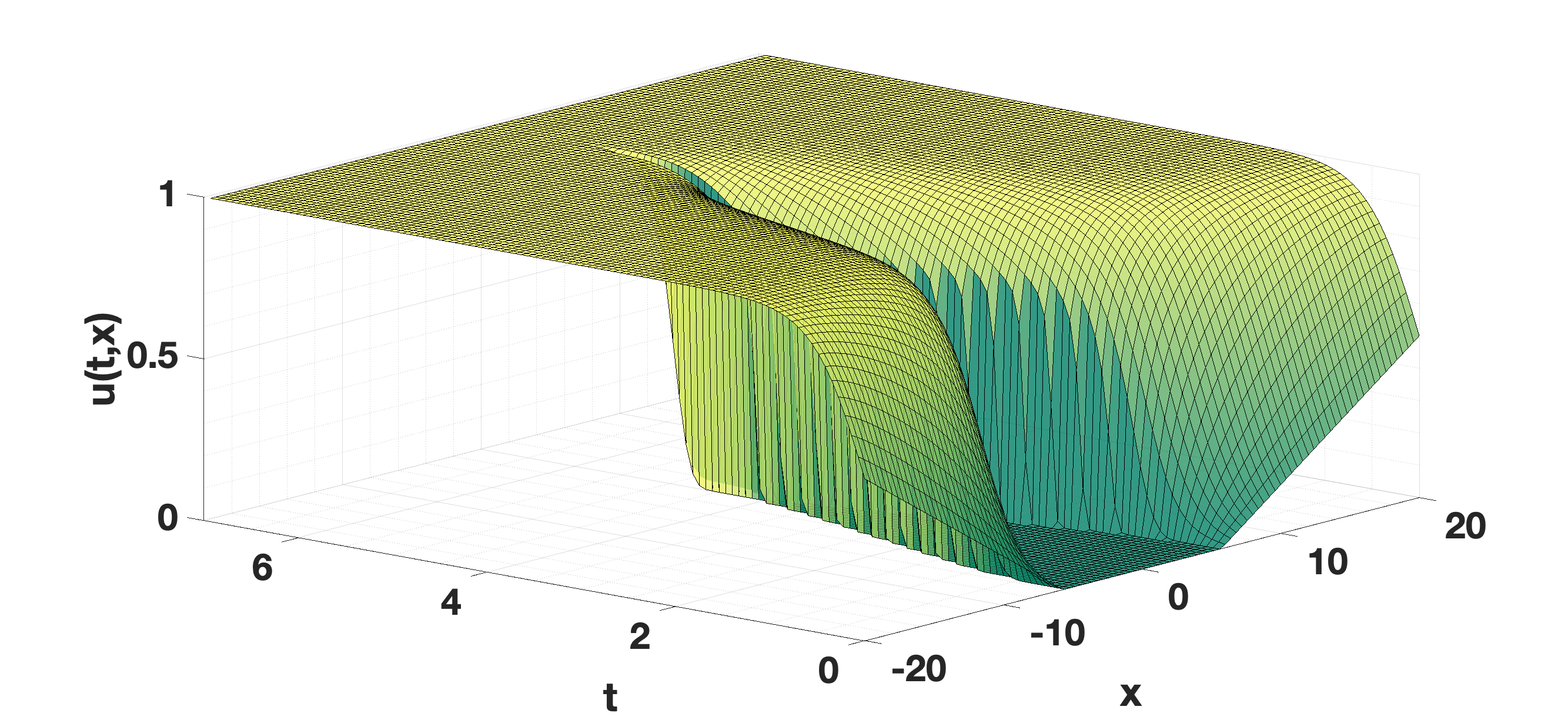}
	\includegraphics[scale=0.17]{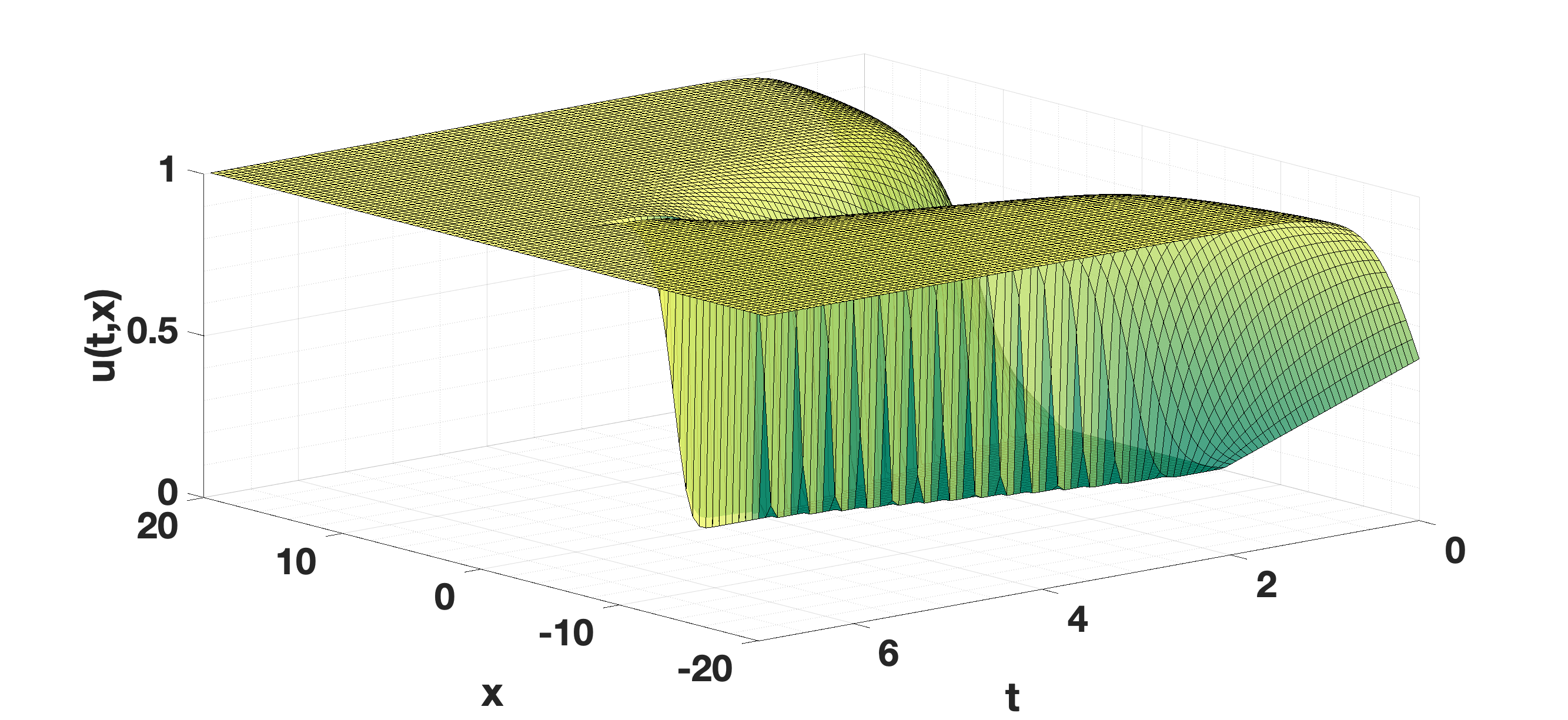}
	\caption{{\textit{In this figure, we plot the solution of the model \eqref{A.1} starting from the initial distribution \eqref{6.2} (with  $\beta=0.07$ and $K=20$).}}}\label{Fig12}
\end{figure}
It is observed that the speed of healing depends strongly on the imperfection of the wound. If we compare the two simulations, we see that the wound seems much larger in Figure \ref{Fig9} than in Figure \ref{Fig11}. But the time required for healing is about $2$ days in Figure \ref{Fig10} whereas it is about $5$ days in Figure \ref{Fig12}. Therefore, the imperfection of the wound has a strong influence on the healing time.

\paragraph{Competing interests:} The authors declare none. 

\appendix
\begin{center}
	{\LARGE	\textbf{Appendix}}
\end{center}
\section{Upwind method applied to the numerical scheme}
\label{SectionA}
In Section \ref{Section5}, we use the following system of PDE to run the numerical simulations  
\begin{equation}\label{A.1}
	\left\{\begin{array}{l}
		\partial_{t} u(t, x)=\chi \partial_{x}\left(u(t, x) \partial_{x} p(t, x)\right)+ \lambda \; u(t, x)(1-u(t, x)/ \kappa ), \quad t\in(0, T], x \in [-K, K], \vspace{0.2cm}\\
		p(t, x)-\sigma^{2} \partial_{x x} p(t, x)=u(t, x), \quad t\in(0, T], x \in [-K, K],\vspace{0.2cm}\\
		\partial_{x} p(t, - K)=\partial_{x} p(t, + K)=0, \quad t\in(0,T],
	\end{array}\right.
\end{equation}
with 
$$
u(t,x)=u_0(x) \in L^\infty_+ \left( [-K,K], \R\right).
$$
Now we use the finite volume method to consider equation \eqref{A.1}. Our numerical scheme reads as follows
\begin{equation}\label{A.2}
	u_{i}^{n+1}=u_{i}^{n}-\chi\frac{\Delta t}{\Delta x}\left(\phi(u_{i+1}^{n},u_{i}^{n})-\phi(u_{i}^{n},u_{i-1}^{n})\right)+\Delta t u_{i}^{n}(1-u_{i}^{n}), \quad i=1, 2, \ldots, M, 
\end{equation}
where the flux $\phi(u_{i+1}^{n},u_{i}^{n})$ for $i=0, \ldots, M$ defined as
\begin{equation}\label{A.3}
	\phi(u_{i+1}^{n},u_{i}^{n})=\left(v_{i+\frac{1}{2}}^{n}\right)^{+}u_{i}^{n}-\left(v_{i+\frac{1}{2}}^{n}\right)^{-}u_{i+1}^{n}=
	\begin{cases}
		v_{i+\frac{1}{2}}^{n}u_{i}^{n}, & v_{i+\frac{1}{2}}^{n}\geq0, \\
		v_{i+\frac{1}{2}}^{n}u_{i+1}^{n}, & v_{i+\frac{1}{2}}^{n}<0,
	\end{cases}
\end{equation}
where 
$$
x^+=\max (0,x), \text{ and }x^-=\max (0,-x),
$$ 
and
\begin{equation}\label{A.4}
	v_{i+\frac{1}{2}}^{n}=-\frac{p_{i+1}^{n}-p_{i}^{n}}{\Delta x}, i=0, 1, \ldots, M,
\end{equation}
where
$$
v_{0+\frac{1}{2}}^{n}=v_{M+\frac{1}{2}}^{n}=0.
$$
Moreover the vector $P^{n}$ is defined by  
\begin{equation}\label{A.5}
	P^{n}:=\left( I- \dfrac{\sigma^2}{\Delta x^2} A \right) ^{-1} U^{n},
\end{equation}
where
\begin{equation}\label{A.6}
	A=
	\begin{pmatrix}
		-1 & 1      &        &        &\\
		1  & -2     & 1      &        &\\
		&\ddots  &\ddots  &\ddots  &\\
		&        & 1      & -2     & 1 \\
		&        &        & 1      & -1
	\end{pmatrix}
	_{M\times M}.
\end{equation}
Indeed, we have
\begin{equation}\label{A.7}
	p_{i}^{n}-\dfrac{\sigma^2}{\Delta x^2} \left(p_{i+1}^{n}-2p_{i}^{n}+p_{i-1}^{n} \right)=u_{i}^{n},\quad i=1, 2, \ldots, M, 
\end{equation}
and since we use the Neumann boundary condition, we must impose 
$$
p_0^n=p_1^n \text{ and } p_M^n=p_{M+1}^n. 
$$
Since the Neumann boundary condition corresponds to a no flux boundary condition, we have
\begin{equation}\label{A.8}
	\begin{aligned}
		\phi(u_{1}^{n},u_{0}^{n})=0,\text{ and }  \phi(u_{M+1}^{n},u_{M}^{n})=0,
	\end{aligned}
\end{equation}
which corresponds to $p_0^n=p_1^n$ and $p_{M+1}^n=p_M^n$. 
Therefore, the numerical scheme at the boundary becomes
\begin{equation}\label{A.9}
	\begin{aligned}
		u_{1}^{n+1}&=u_{1}^{n}-\chi\frac{\Delta t}{\Delta x}\phi(u_{2}^{n},u_{1}^{n})+\Delta t \, u_{1}^{n}(1-u_{1}^{n}),\\
		u_{M}^{n+1}&=u_{M}^{n}+\chi\frac{\Delta t}{\Delta x}\phi(u_{M+1}^{n},u_{M}^{n})+\Delta t \, u_{M}^{n}(1-u_{M}^{n}).
	\end{aligned}
\end{equation}
Due to the boundary condition, we have the conservation of mass for equation \eqref{A.1} when the reaction term equals zero.

\end{document}